\documentclass[12pt,A4,reqno]{amsart}
\usepackage{amsfonts}
\usepackage{mathrsfs}

\usepackage{amssymb}
\usepackage{amsmath,amscd}
\usepackage{color}
\usepackage{epsf}
\usepackage{graphicx}
\usepackage{xypic}

\theoremstyle{plain}
\newtheorem{thm}{Theorem}[section]
\newtheorem{lem}[thm]{Lemma}

\theoremstyle{definition}
\newtheorem{defi}[thm]{Definition}
\newtheorem{rem}[thm]{Remark}
\newtheorem{exam}{Example}

\newcommand{\R}{\mathbb R}
\newcommand{\Z}{\mathbb Z}

\begin{document}

\title [\ ] {On the constructions of free and locally standard $\Z_2$-torus actions on Manifolds}

\author{Li Yu}
\address{Department of Mathematics and IMS, Nanjing University, Nanjing, 210093, P.R.China}
\email{yuli@nju.edu.cn}
\thanks{This work is partially supported by a grant from NSFC
(No.10826040) and by the Scientific Research Foundation for the
Returned Overseas Chinese Scholars, State Education Ministry}




\begin{abstract}
    We introduce an elementary way of constructing
    principal $(\Z_2)^m$-bundles over compact smooth manifolds. In addition,
     we will define a general notion of
    locally standard $(\Z_2)^m$-actions on closed
    manifolds for all $m\geq 1$, and then give a
   general way to construct all such $(\Z_2)^m$-actions from the orbit space.
   Some related topology problems are also studied.
 \end{abstract}

\maketitle

  \section{Introduction}

      If the group $(\Z_2)^m$ acts freely and smoothly on a closed manifold
      $M^n$, the orbit space $Q^n$ is also a closed manifold.
      We can think of $M^n$ either as a principal
      $(\Z_2)^m$-bundle over $Q^n$ or as a regular covering over
      $Q^n$ with deck transformation group $(\Z_2)^m$.
      In algebraic topology, we have a standard way to
      recover $M^n$ from $Q^n$ using the universal covering space
      of $Q^n$ and the monodromy map of the covering (see~\cite{Hatcher02}).
      However, it is not very easy to visualize
      the total space of the covering using that construction.
      Considering the speciality of $(\Z_2)^m$, it is desirable to
      have a new way of constructing such regular coverings from the orbit spaces which can
      really help us to visualize the total space more easily. In this paper,
      such a construction will be given with the name \textit{glue-back construction}.\vskip .2cm

      Another source of nice $\Z_2$-torus actions on manifolds are
      locally standard actions (see~\cite{DaJan91}).
      Suppose $M^n$ is a closed manifold with a smooth locally
       standard $(\Z_2)^n$-action, let $X^n =M^n \slash (\Z_2)^n$ be the orbit
        space and $\pi : M^n \rightarrow X^n$ be the orbit map.
         It is
        well known that $X^n$ is a nice $n$-manifold with corners,
        and if the action is not free, $X^n$ will have boundary.
       The $(\Z_2)^n$-action determines a \textit{characteristic
       function} $\nu_{\pi}$ (taking values in $(\Z_2)^n$) on the facets of
      $X^n$, which encodes the information of isotropy subgroups of
      the non-free orbits. In particular, when $X^n$ is a convex simple polytope,
      there is a standard construction to recover $M^n$ (up to
      equivariant homeomorphism) from the characteristic function $\nu_{\pi}$ on $X^n$ (see~\cite{DaJan91}).
      But in general, if $H^1(X^n,\Z_2)$ is not trivial, we need an
      additional piece of data to recover $M^n$ --- a principal
      $(\Z_2)^n$-bundle $\xi_{\pi}$ over $X^n$ which encodes the
      information of the free orbits of the action (see~\cite{MaLu08}).
      However, the bundle information $\xi_{\pi}$ has a quite different flavor from
       the characteristic function $\nu_{\pi}$ and is not so easy to be visualized in
      the orbit space $X^n$.
      In this paper, we will combine the characteristic function $\nu_{\pi}$
      and the $(\Z_2)^n$-bundle $\xi_{\pi}$ on $X^n$ into a composite
      $(\Z_2)^n$-valued colorings $(\lambda_{\pi}, \mu_{\pi})$ on a new manifold $U^n$
      (called a \textit{$\Z_2$-core} of $X^n$),
      which is a nice manifold with corners obtained from $X^n$ (but not uniquely).
      And up to equivariant homeomorphisms, we can recover $M^n$ from the composite
      $(\Z_2)^n$-coloring $(\lambda_{\pi},\mu_{\pi})$ on $U^n$ from a
      generalized glue-back construction.
       \vskip .2cm

     Moreover, we can define a general notion of
    locally standard $(\Z_2)^m$-action on $n$-dimensional
    manifolds for all $m\geq 1$, which includes all
    free $(\Z_2)^m$-actions on $n$-dimensional manifolds. The glue-back construction can be
    applied in this general setting as well. So actually we do
    not assume $m=n$ at all in this paper.
     \vskip .2cm

          The paper is organized as following.
          In section~\ref{Sec2}, we will explain
           how to get a $\Z_2$-core $V^n$
           from a closed manifold $Q^n$ and introduce an important structure
           on $V^n$ called \textit{involutive panel structure}.
           We will introduce several definitions
           concerning this structure to make our subsequent discussions precise and convenient.
           Some explicit examples will be analyzed to illustrate these
           definitions.
           In section~\ref{Sec3}, we will introduce the glue-back construction
            from a $\Z_2$-core $V^n$ of $Q^n$ with a $(\Z_2)^m$-colorings.
            And we will show that any principal $(\Z_2)^m$-bundles over
             $Q^n$ can be obtained in this way. Also
             the glue-back construction makes sense
             for any nice manifold with corners equipped with an involutive panel
             structure. Some properties of this construction will be
             studied along with some explicit examples.
          In Section~\ref{Sec4}, we will
          generalize the notion of $\Z_2$-core and glue-back construction to compact manifolds
          with boundary as well.
         Then in Section~\ref{Sec5}
          we define a general notion of locally standard
          $(\Z_2)^m$-actions on closed $n$-manifolds for any $m\geq 1$ and
          apply the glue-back construction
          to this general setting.
          Especially, the notion of involutive panel structure
         is used to unify all our constructions.
            In addition, we will state some classification theorems of
           locally standard $(\Z_2)^m$-actions
           on closed $n$-manifolds up to (weak) equivariant homeomorphisms.
          In Section~\ref{Sec6}, we will
          discuss how to get some topological information (e.g. the number of connected
          components and orientability)
        of the glue-back construction of locally standard $(\Z_2)^m$-actions from the
          $(\Z_2)^m$-colorings. In the end, we will propose some problem for the further study.
          \vskip .2cm

     The main idea of the paper is inspired by the description of locally standard
       $\Z_2$-torus manifolds in~\cite{MaLu08}. An aim of this paper
       is to establish a framework for studying general locally standard
      $(\Z_2)^m$-actions on $n$-manifolds in the future. In
      particular, the author will use the glue-back construction
      to study the Halperin-Carlsson conjecture for
      free $(\Z_2)^m$-actions on compact manifolds in a sequel paper.
      Also, the involutive panel structure defined in this paper
      might have some independent value.\vskip .2cm

         In this paper, we denote the quotient group $\Z\slash 2\Z$ by $\Z_2$
     and always think of $(\Z_2)^m$ as an additive group.
     In addition, we will use the following conventions: \vskip .1cm
       \begin{enumerate}
         \item any manifold and submanifold in this paper is smooth; \vskip .1cm
         \item we always identify an embedded submanifold with its
         image; \vskip .1cm
         \item any $(\Z_2)^m$-actions on manifolds in this paper are
         smooth and effective.
           \\
       \end{enumerate}

  \section{$\Z_2$-core of a closed manifold and Involutive panel structure}\label{Sec2}

    Suppose $M^n$ is an $n$-dimensional closed manifold with a free
     $(\Z_2)^m$-action ($m$ is an arbitrary positive integer),
     let $Q^n =M^n \slash (\Z_2)^m$ be the orbit
     space and $\pi : M^n \rightarrow Q^n$ be the orbit map.
     Then $Q^n$ is also a closed manifold. In addition,
     we always assume $Q^n$ is connected in this paper.
     \vskip .2cm

      We can consider the orbit map
       $\pi : M^n \rightarrow Q^n$ either as a regular $(\Z_2)^m$ covering
       or as a principal $(\Z_2)^m$-bundle map. Note that $M^n$
       may not be connected in general.
       \vskip .2cm

      It is well-known that the up to bundle isomorphism,
        principal $(\Z_2)^m$-bundles over $Q^n$ are one-to-one correspondent
        with elements of $ H^1(Q^n,(\Z_2)^m) $. Then $\pi: M^n \rightarrow Q^n$
        determines an element
            \begin{equation*}
          \Lambda_{\pi} \in H^1(Q^n,(\Z_2)^m) \cong
          \mathrm{Hom}(H_1(Q^n,\Z_2),(\Z_2)^m).
             \end{equation*}

        From another viewpoint, as a regular covering space, $\pi: M^n \rightarrow Q^n$
         is determined by its monodromy map $\mathcal{H}_{\pi} : \pi_1(Q^n) \rightarrow
           (\Z_2)^m$. Since $(\Z_2)^m$ is an abelian group, we get an induce group
           homomorphism
           $\mathcal{H}^{ab}_{\pi} : H_1(Q^n,\Z_2) \rightarrow
           (\Z_2)^m$ which is exactly the $\Lambda_{\pi}$ above.
           Moreover, by the Poincar\'e duality, we have
            $H_{n-1}(Q^n,\Z_2) \cong H_1(Q^n,\Z_2)$. So
             we obtain a group homomorphism
           $\Lambda^*_{\pi} : H_{n-1}(Q^n,\Z_2) \rightarrow
           (\Z_2)^m$. \vskip .2cm

     The above analysis suggests us to construct a new geometric object from $Q^n$ which can carry
     all the information of $\Lambda_{\pi}$ (or $\Lambda^*_{\pi}$).
      First, we recall a well-known theorem in algebraic topology.
      \vskip .2cm

      \begin{thm}[Hopf]
       Let $f: M^m \rightarrow N^n$ be a smooth map
      between closed oriented manifolds and $L^{n-1}\subset N^n$ a
      closed, oriented submanifold of codimension $p$ such
     that $f$ is transverse to $L$. Write $u \in H^p(N)$ for the Poincar\'e dual
    of $[L]_N$, that is, $u\cap [N] = [L]_N$. Then $[f^{-1}(L)]_M = f^*(u)\cap [M]$.
     In other words: If $u$ is Poincar\'e dual to $[L]_N$, then $f^*(u) \in H^p(M)$
      is Poincar\'e dual to $[f^{-1}(L)]_M$. If using $\Z_2$ coefficient,
       we do not need to assume that $M^m$ and $N^n$ are orientable.
       \end{thm}
       \begin{proof}
         Use the naturality of the Thom class of the tangent bundle.
       \end{proof} \vskip .2cm

     If $H^1(Q^n,\Z_2) =0$, then any principal $(\Z_2)^m$-bundle over $Q^n$
     is trivial. So in the rest of this paper, we always assume $H^1(Q^n,\Z_2) \neq
     0$. let $\{ \varphi_1,\cdots,\varphi_k \} $ be a basis of
      $H^{1}(Q^n,\Z_2)$, and let
      $\{ \alpha_1,\cdots,\alpha_k \}$ be the basis of $H_{n-1}(Q^n,\Z_2)$
      that is dual to $\{ \varphi_1,\cdots,\varphi_k \}$ under
      the Poincar\'e duality.

     \begin{lem} \label{Lem:Sig}
       $\alpha_1,\cdots,\alpha_k$ can be
        represented by codimension one connected embedded
         submanifolds of $Q^n$.
     \end{lem}
     \begin{proof}
         Since $H^1(Q^n,\Z_2)\cong [Q^n, K(\Z_2,1)] = [Q^n, \R
         P^{\infty}] = [Q^n,\R P^{n+1}]$, an element $\varphi\in H^1(Q^n,\Z_2)$
         corresponds to a homotopy class of maps $[f]: Q^n
         \rightarrow \R P^{n+1}$ such that
         $\varphi= f^*(\Phi)$ where $\Phi$ is a generator of $H^1(\R P^{n+1},\Z_2)$.
          Then the $\Z_2$-homology class represented by a canonically embedded
          $\R P^n\subset \R P^{n+1}$ is the Poincar\'e
         dual of $\Phi$. We can always assume that $f$ is smooth and transverse
         to $\R P^n$. Then by above theorem, $\Sigma = f^{-1}(\R P^n)$ is
         a codimension $1$ embedded submanifold in $Q^n$ which
         is Poincar\'e dual to $\varphi$.
         So we can find codimension one embedded submanifolds
         $\Sigma_1,\cdots,\Sigma_k$ which are
         Poincar\'e dual to $\varphi_1,\cdots, \varphi_k$ respectively.
         In addition, we can always choose $\Sigma_1,\cdots,\Sigma_k$
         to be connected. Indeed, for $n=1,2$, this is obviously
         true. And when $n\geq 3$, if some $\Sigma_i$ is not connected,
       we can connect all its components via thin tubes in $Q^n$,
       which will not change the homology class of $\Sigma_i$ in $H_{n-1}(Q^n,\Z_2)$.
     \end{proof}

      \vskip .2cm

         A collection of codimension-one embedded closed submanifolds
           $\{ \Sigma_1,\cdots,\Sigma_k \}$ is called a \textit{$\Z_2$-cut
        system} of $Q^n$ if they satisfy the following conditions:
             \begin{enumerate}
               \item the homology classes $ [\Sigma_1],\cdots,[\Sigma_k] $
                     form a $\Z_2$-linear basis of $H_{n-1}(Q^n,\Z_2)$.
                      \vskip .1cm

               \item  $ \Sigma_1,\cdots,\Sigma_k $ are in general
               position in $Q^n$ which means that: \vskip .1cm
              \begin{enumerate}
                   \item $\Sigma_1, \cdots, \Sigma_k$ intersect transversely with each
                    other and, \vskip .1cm

                 \item if $\Sigma_{i_1} \cap \cdots \cap \Sigma_{i_s}$ is not empty,
                          then it is an embedded submanifold
                         of $Q^n$ with codimension $s$.\vskip .2cm
           \end{enumerate}
             \end{enumerate}

        Now we choose a small tubular
        neighborhood $N(\Sigma_i)$ of each $\Sigma_i$ in $Q^n$,
        and then remove the interior of each $N(\Sigma_i)$ from $Q^n$. The
        manifold that we get is:
           $$V^n = Q^n - \bigcup^{k}_{i=1} int(N(\Sigma_i))  $$
           which is called a \textit{$\Z_2$-core} of $Q^n$
        from cutting $Q^n$ open along $\Sigma_1,\cdots,\Sigma_k$.
         The boundary of $V^n$ is $\partial \left( \bigcup_i  N(\Sigma_i)  \right)$. We call
          $\partial N(\Sigma_i)$ the \textit{cut section} of $\Sigma_i$ in $Q^n$. \vskip
         .2cm

         Notice that the projection $\eta_i : \partial N(\Sigma_i) \rightarrow
        \Sigma_i$ is a double cover, either trivial or nontrivial.
        Let $\overline{\tau}_i$ be the generator of
       the deck transformation of $\eta_i$.
       Then $\overline{\tau}_i$ is a free involution on $\partial N(\Sigma_i)$, i.e.
        $\overline{\tau}_i$ is a homeomorphism with no fixed point
         and $\overline{\tau}^2_i = id$. \vskip .2cm

        The boundary of $V^n$
        is tessellated by $(n-1)$-dimensional compact connected manifolds (with boundary) called
        {\textit{facets} of $V^n$.
        Any connected component of the intersection of some facets is called
        a \textit{(closed) face} of $V^n$.
         Since $\Sigma_1,\cdots,\Sigma_k$ are in general position in $Q^n$,
       so $V^n$ is a \textit{nice manifold with corners},
       which means that each codimension $l$ face of $V^n$ is the
        intersection of exactly $l$ facets.
        For a comprehensive introduction of manifolds with corners and related concepts,
        see~\cite{Janich68} and ~\cite{Da83}. \vskip .2cm

        \begin{rem}
          $V^n$ might not have
        \textit{vertices} ($0$-dimensional strata) on the
        boundary.  for example, if
        $Q^n=S^{n-1}\times S^1$ ($n\geq 3$), cutting $Q^n$ along
        $S^{n-1}\times \{ 1 \}$ gives a $\Z_2$-core $V^n = S^{n-1} \times [0,1] $ of $Q^n$
        whose boundary
        consists of two disjoint $S^{n-1}$. \vskip .2cm
         \end{rem}

       In addition, we call the union of facets of $V^n$ that
       belong to $\partial N(\Sigma_i)$ a \textit{panel},
       denoted by $P_i$ (see Figure~\ref{p:Panel_Struc}).
      So $\{ P_1,\ldots, P_k \}$ forms a panel structure on $V^n$.
      Recall that a \textit{panel structure} on a topological space $Y$ is a locally finite
      family of closed subspaces $\{ Y_{\alpha} \}_{\alpha \in \mathcal{A}}$
       indexed by some set $\mathcal{A}$. Each $Y_{\alpha}$ is called
       a \textit{panel} of $Y$ (see~\cite{Da83}). \vskip .2cm

         Notice that
         the involution $\overline{\tau}_i$ may not
         map a $P_i \subset \partial N(\Sigma_i)$ into $P_i$.
         This is because that there might be some $N(\Sigma_j)$ so that
          $\overline{\tau}_i$ and $\overline{\tau}_j$
        do not commute at the intersections $\partial N(\Sigma_i) \cap \partial
        N(\Sigma_j)$ (see the left picture in
        Figure~\ref{p:Deform_Invol}). But the following lemma shows that we can always deform
        the $\overline{\tau}_i$ and $\overline{\tau}_j$ locally by
        isotopies
        to make them commute at $\partial N(\Sigma_i) \cap \partial
        N(\Sigma_j)$.  \vskip .2cm

       \begin{figure}
         \includegraphics[width=0.65\textwidth]{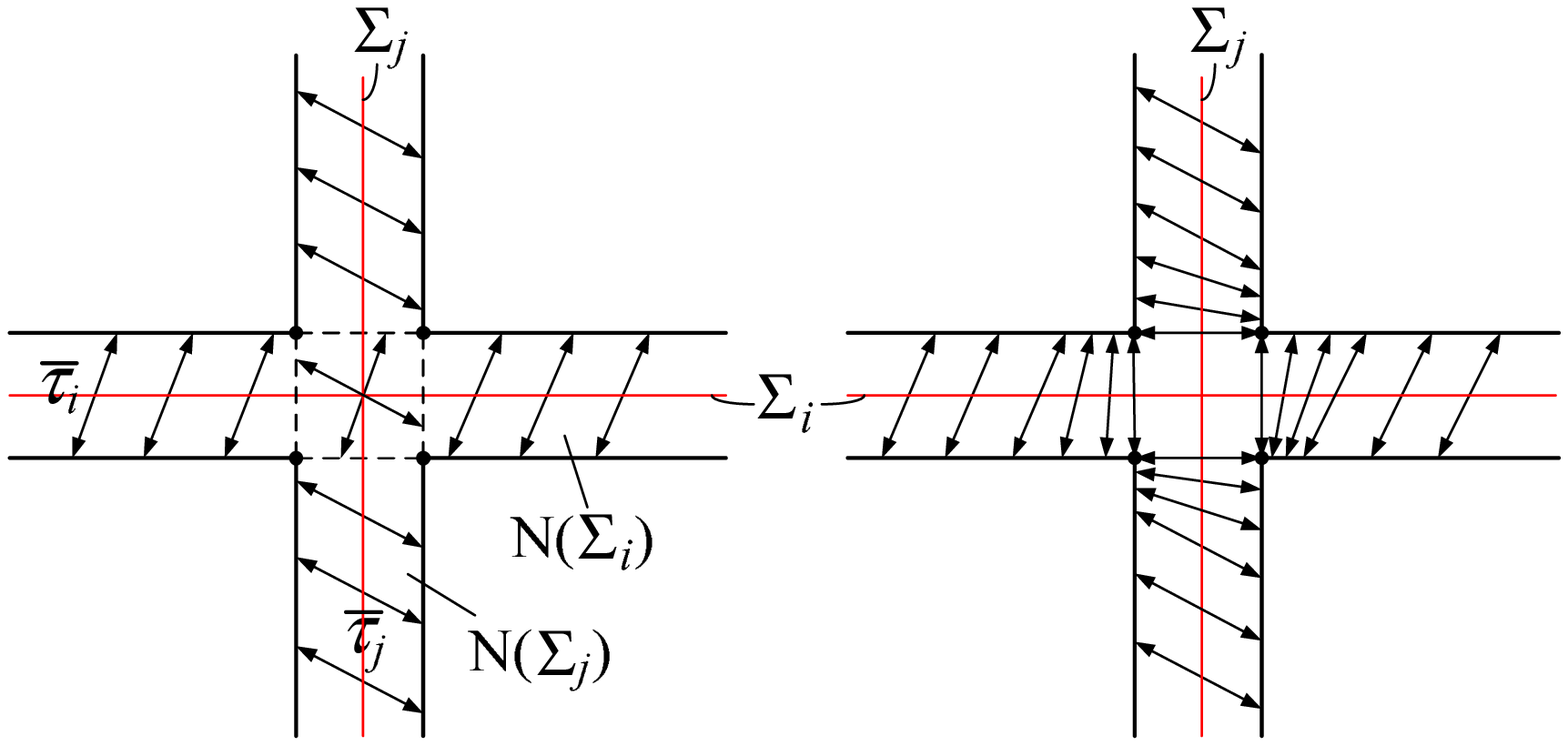}
          \caption{ Local deformation of $\overline{\tau}_i$'s }\label{p:Deform_Invol}
      \end{figure}

         \begin{lem} \label{Lem:Deform_Invol}
           We can deform $\overline{\tau}_i$'s around the intersections of
          $\partial N(\Sigma_i)$'s so that after the deformations, we have:
          \begin{itemize}
            \item[(i)] each $\overline{\tau}_i$ is still a free involution $\partial
            N(\Sigma_i) \rightarrow \partial N(\Sigma_i)$ and the
            quotient $\partial N(\Sigma_i) \slash \langle x \sim \overline{\tau}_i (x) \rangle
            \cong \Sigma_i $; \vskip .1cm

            \item[(ii)] for any $1\leq i,j \leq k$,
         $\partial N(\Sigma_i) \cap \partial N(\Sigma_j)$
          becomes an invariant set of both $\overline{\tau}_i$ and $\overline{\tau}_j$;
           \vskip .1cm

            \item[(iii)] for any point $x\in \partial N(\Sigma_i) \cap \partial
          N(\Sigma_j)$, $\overline{\tau}_i (\overline{\tau}_j(x)) = \overline{\tau}_j(\overline{\tau}_i(x))$.
          \end{itemize}
         \end{lem}

        \begin{proof}
          For $\forall\, p \in \Sigma_i$, let $T_p \Sigma_i$ be the
          tangent plane of $\Sigma_i$ at $p$ in $Q^n$. Suppose
          $\Sigma_{i_1} \cap \cdots \cap \Sigma_{i_s}$ is nonempty.
           For any $p \in \Sigma_{i_1}\cap \cdots \cap \Sigma_{i_s}$,
           there exists an open neighborhood $U$ of $p$ and
           a homeomorphism $\phi: U \rightarrow \R^n$ such that
           $\phi(p) = 0$ and for any $i\in \{ i_1,\cdots, i_s \}$, we have
              \begin{itemize}
                \item $\phi(\Sigma_i \cap U)$ is the
                coordinate hyperplane $H_i=\{ (x_1,\cdots, x_n) \in \R^n \, | \, x_i=0
                \}$. So $\phi(\Sigma_{i_1} \cap \cdots \cap \Sigma_{i_s} \cap U)=
                 H_{i_1} \cap \cdots \cap H_{i_s}$. \vskip .1cm
                \item $\phi (N(\Sigma_i)\cap U) = H_i \times [-1,1] \subset \R^n$.\vskip .1cm

              \item in the chart $(U,\phi)$, $\overline{\tau}_i$ defines a homeomorphism
           $f_i : H_i\times \{1\} \rightarrow H_i \times \{ -1\}$.
              \end{itemize}

            Then we can deform these $\overline{\tau}_{i}$ via some isotopies in $U$ such that
            they satisfy our requirements (i) -- (iii) in
            $U$. Indeed, let $r_i$ be the reflection of $\R^n$ about the hyperplane $H_i$.
            Then we can isotope $f_{i}$ such that for any
            $ j\in \{ i_1,\cdots, i_s \}$ with $j\neq i$, we have
            \begin{equation} \label{Equ:Deform}
              f_i(x) = r_i(x),\ \text{for any}
              \ x \in  H_j\times \{ \pm 1 \} \cap H_i \times \{ 1\}.
            \end{equation}
             Then these $\overline{\tau}_i$'s obviously meet our requirements.
            Moreover, since the (i) -- (iii) are coordinate-independent
            properties, we can carry out the
           deformations of these $\overline{\tau}_i$'s chart by chart around
          $\Sigma_{i_1}\cap \cdots \cap \Sigma_{i_s}$ until the (i) -- (iii) are satisfied at all
          places.
          In addition,
           we should do the deformation of the $\overline{\tau}_i$'s
          in the charts around the
            higher degree intersection points first, then extend to the charts around
             lower degree intersection points. In the end, we will get $\overline{\tau}_i$'s
          which satisfy all the requirements (i) -- (iii). We remark
          that doing the isotopy of these $\overline{\tau}_i$ in a chart might
          alter what we have previous done in another chart, but
          since the (i) -- (iii) are coordinate-independent
            properties, the altering will not cause any inconsistency
            in our construction.
        \end{proof} \vskip .2cm

        After the local deformations of $\overline{\tau}_i$'s described in the preceding lemma,
        the restriction of
           each $\overline{\tau}_i$ on $P_i \subset \partial N(\Sigma_i)$ defines
          a free involution on $P_i$, denoted by $\tau_i$.
          Because of the existence of these $\tau_i$,
           we call the set of panels of $V^n$ an
      \textit{involutive panel structure}. We will always assume that $V^n$ has this
      involutive panel structure in the rest of the paper. Note that
      $\{ \tau_i : P_i \rightarrow P_i \}_{1\leq i \leq k}$
      satisfy: \vskip .1cm

    \begin{itemize}
      \item $\tau_i$ maps a face $f$ of $P_i$ to a face $f'$ of $P_i$ (it is possible that
      $f'=f$ though);
      \vskip .1cm

      \item $\tau_i (P_i \cap P_j) \subset P_i \cap
             P_j$ for all $1\leq i,j \leq k$; \vskip .1cm
      \item $\tau_i\circ\tau_j = \tau_j\circ \tau_i
                 : P_i \cap P_j \rightarrow P_i \cap P_j$ for all
                 $1\leq i,j \leq k$.
    \end{itemize}

      \vskip .2cm

     \begin{figure}
            \includegraphics[width=0.41\textwidth]{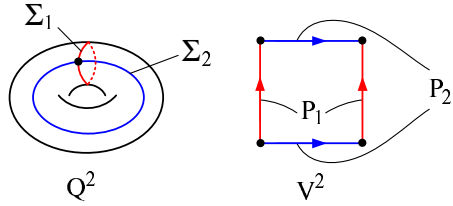}
               \caption{A $\Z_2$-core of torus}
               \label{p:Panel_Struc}
            \end{figure}

       For any $I= \{ i_1 , \cdots, i_s \} \subset \{ 1,\cdots, k \}$ with $|I|=s\geq 1$,
       we define:
       \begin{equation}\label{Equ:subpanel}
         P_{I} := P_{i_1} \cap \cdots \cap P_{i_s} \subset V^n,
         \quad  \Sigma_I := \Sigma_{i_1}\cap \cdots \cap \Sigma_{i_s} \subset Q^n.
       \end{equation}
      When $I=\varnothing$, we define $P_{\varnothing} = V^n$ and $\Sigma_{\varnothing} =Q^n$.
       If $P_I$ with $|I|\geq 2$ is nonempty, it is called a
     \textit{subpanel} of $V^n$.  Notice that the $P_I$ is empty whenever $|I| > n$.
     Although $Q^n$ is assumed to be connected, the $\Sigma_I$ may
     not be connected.
      \vskip .2cm

          For any point $x \in P_i$, let $x^*_{P_i} = \tau_i(x) \in P_i$.
               We call $x^{*}_{P_i}$ the \textit{twin point} of $x$ in $P_i$.
           Obviously, $x^*_{P_i}\neq x$ since $\tau_i$ here is free.\vskip .2cm

       Generally, for a face $f \subset P_i$, if the face $f^*_{P_i} = \tau_i(f)$
       is disjoint from $f$, it is called
       the \textit{twin face} of $f$ in $P_i$.
        Otherwise, $f$ is called \textit{self-involutive} in $P_i$ in the sense
          that $\tau_i(f)=f$.
            In particular,  if a facet $F$ of $V^n$ is not self-involutive, it has a unique
            twin facet $F^*$ which belongs to the same panel as $F$. We call
          $\widehat{F} :=  F \cup F^*$ a \textit{facet pair}.
         \vskip .2cm

     As an embedded submanifold of $Q^n$, $\Sigma_i$
       could be two-sided or one-sided, which is determined by the orientability
       of the normal bundle of $\Sigma_i$ in $Q^n$. If $\Sigma_i$ is two-sided,
       any facet $F$ in $P_i$ has a twin facet $F^*$ (see
       Figure~\ref{p:Panel_Struc}).
       But if $\Sigma_i$ is one-sided, some facet in $P_i$ might be self-involutive (see
       Figure~\ref{p:RP2}).
     \vskip .2cm
             \begin{figure}
                \includegraphics[width=0.54\textwidth]{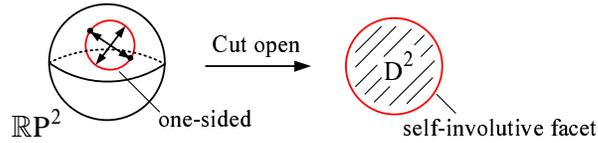}
                \caption{A $\Z_2$-core of real projective plane}\label{p:RP2}
           \end{figure}

     \begin{rem}
       The facets in the same panel of $V^n$ are pairwise disjoint since each $\Sigma_i$
           in the $\Z_2$-cut system has no self-intersections. But
           a panel of $V^n$ may consist of more than one facet pair
        (see the Example~\ref{Exam:Z2core-Small-Cover} below).
        \end{rem}
     \vskip .2cm

      If we identify any points of $V^n$ with all their
       duplicate points in $V^n$, we will get a manifold denoted by
       $\widehat{Q}^n$. Let $\varrho: V^n \rightarrow \widehat{Q}^n$ be the quotient map.
       \vskip .2cm

       \begin{lem} \label{Lem:Quotient-Homeo}
       There exists a homeomorphism $h: \widehat{Q}^n \rightarrow
       Q^n$ with $h(\varrho(P_I)) = \Sigma_I$ for any
        $I\subset \{ 1,\cdots, k\}$.
       \end{lem}
      \begin{proof}
           By our construction of $\tau_i$, it is easy to see that
           $\varrho(P_i) \cong \Sigma_i$ for any $1\leq i \leq k$.
           In addition, there exists a neighborhood
           $N(\partial V^n)$ of $\partial V^n$ in $V^n$ with
            $N(\partial V^n) \cong \partial V^n \times [0,\varepsilon]$
           so that
           $\varrho(N(\partial V^n)) \subset \widehat{Q}^n$ is homeomorphic to
           $\overset{k}{\underset{i=1}{\bigcup}} N(\Sigma_i) \subset Q^n$.
            Let
           $U^n= Q^n \backslash int(\overset{k}{\underset{i=1}{\bigcup}} N(\Sigma_i))$. Then we can
           think of $\widehat{Q}^n$ (or $Q^n$) as the gluing of
           $N(\partial V^n)$  $\left( \text{or}\ \overset{k}{\underset{i=1}{\bigcup}}
           N(\Sigma_i)\right)$ with $U^n$ along their boundary, that is:

           $ \qquad \qquad \widehat{Q}^n = \varrho(N(\partial V^n)) \bigcup_{\varphi_1}
           U^n,\ \ Q^n = \left(\overset{k}{\underset{i=1}{\bigcup}} N(\Sigma_i) \right) \bigcup_{\varphi_2} U^n
           $,

           where the $\varphi_1: \partial(\varrho(N(\partial V^n))) \rightarrow \partial U^n $ and
           $\varphi_2: \partial \left( \overset{k}{\underset{i=1}{\bigcup}} N(\Sigma_i) \right)
            \rightarrow \partial U^n$ are homeomorphisms. If we
            identify $\partial(\varrho(N(\partial V^n)))$ with
             $\partial \left( \overset{k}{\underset{i=1}{\bigcup}} N(\Sigma_i)
             \right)$, $\varphi_1$ and $\varphi_2$ are actually
             isotopic because the local deformations we
             make on the $\tau_i$'s are all isotopies of homeomorphisms. So we can construct a homeomorphism
             $h: \widehat{Q}^n \rightarrow Q^n$ from an isotopy between $\varphi_1$
             and $\varphi_2$, which satisfies our requirement.
        \end{proof}
        \vskip .2cm

      We call $\rho = h\circ \varrho : V^n \rightarrow Q^n$
      the \textit{restoring map} of $V^n$. Then
       $P_I= \rho^{-1}(\Sigma_I)$ for any $I \subset \{ 1,\cdots, k \}$.
       Obviously, for any point $x$ in the relative interior
       of $P_{i_1} \cap \cdots \cap P_{i_s}$, we have:
       $$ \rho^{-1} (\rho(x)) = \{
         \tau^{\varepsilon_s}_{i_s} \circ \cdots \circ \tau^{\varepsilon_1}_{i_1} (x)\, ; \,
        \varepsilon_j \in \{ 0 ,1 \},  1\leq j \leq s   \}, $$
         where $\tau_{i_j}^0 := id$.
        It is easy to see that $\rho^{-1}(\rho(x))$ consists of
        exactly $2^{s}$ different
       points in $V^n$. Any point $x'\in \rho^{-1} (\rho(x))$ (including $x$ itself) is called a
       \textit{duplicate point} of $x$ in $V^n$.\vskip .2cm

   \begin{exam} \label{Exam:Z2core-Small-Cover}
     Suppose $Q^n$ is a small cover over some simple polytope. It is well known that the
     $\Z_2$-homology classes of $Q^n$ can all be represented by some special embedded submanifolds
     of $Q^n$, called \textit{facial submanifolds} (see~\cite{DaJan91} and~\cite{BP02}).
     And cutting $Q^n$ open along a collection of facial submanifolds of $Q^n$
     will give us a connected $\Z_2$-core $V^n$ of $Q^n$.
    Figure~\ref{p:Panel_biject_3} shows such an example in dimension
     $2$ where $Q^2= \R P^2 \# \R P^2 \# \R P^2$
      is a small cover over a pentagon. A $\Z_2$-core of $Q^2$ is an
      octagon where the four edges marked by\textquotedblleft \textbf{A}\textquotedblright
      \,belong to the same panel.  \vskip .2cm
         \begin{figure}
         \includegraphics[width=0.38\textwidth]{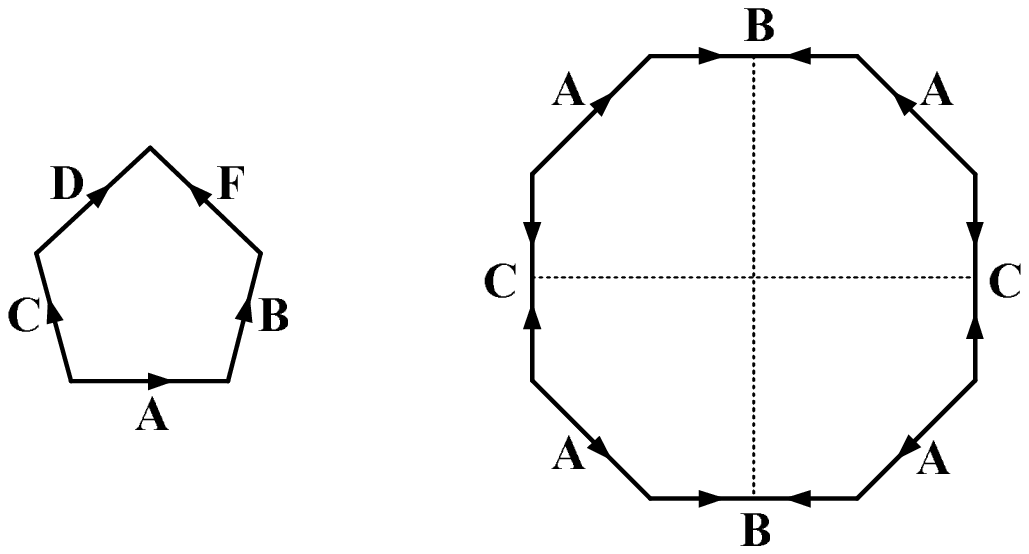}
          \caption{A $\Z_2$-core of $\R P^2 \# \R P^2 \# \R P^2$ }\label{p:Panel_biject_3}
      \end{figure}
  \end{exam}

     \begin{rem}
           A closed connected manifold $Q^n$ may have a $\Z_2$-core $V^n$ with
           $H^1(V^n, \Z_2)\neq 0 $. For example, the $\Z_2$-core of $Q^2=T^2 \# T^2$ shown
           in Figure~\ref{p:g2-surface} is homeomorphic to an annulus.
            \vskip .4cm

          \begin{figure}
            \includegraphics[width=0.48\textwidth]{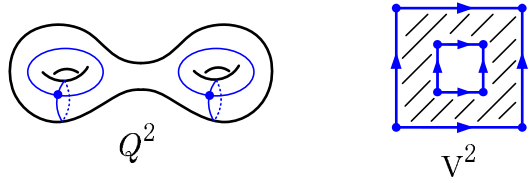}
               \caption{A $\Z_2$-core of $T^2\# T^2$}\label{p:g2-surface}
        \end{figure}
     \end{rem}

      Next, we define a general notion of involutive panel structure
         among all nice manifolds with corners. The involutive panel structure
         on a $\Z_2$-core $V^n$ constructed above is just a special case of this
         general notion.\vskip .2cm

        \begin{defi}[Involutive Panel Structure]
        \label{Def:Invol_Panel}
          Suppose $W^n$ is a nice manifold with corners (may not be connected).
          Suppose the boundary of $W^n$ is the union of several panels
         $P_1,\cdots, P_k$ which satisfy the following
         conditions:
           \begin{itemize}
             \item[(a)] each panel $P_i$ is a disjoint union of facets of $W^n$ and
               each facet is contained in exactly one panel; \vskip .1cm

             \item[(b)] there is an involution
                  $\tau_i$ on each $P_i$ (i.e. $\tau_i$ is a homeomorphism with
                   $\tau^2_i = id_{P_i}$) which
                  sends a face $f \subset P_i$ to a face $f'\subset P_i$
                 (it is possible that $f'=f$); \vskip .1cm

             \item[(c)] for all $i\neq j$, $\tau_i (P_i \cap P_j) \subset P_i \cap
             P_j$ and $\tau_i\circ\tau_j = \tau_j\circ \tau_i
                 : P_i \cap P_j \rightarrow P_i \cap P_j$. \vskip .1cm
            \end{itemize}

            Then we say that $W^n$ has an \textit{involutive panel structure}
            defined by $\{ P_i, \tau_i \}_{1\leq i \leq k}$ on the boundary.
            Note here, we do not require that the involution $\tau_i$ on $P_i$ is free.
        \end{defi}

          Similar to the $\Z_2$-core $V^n$, for any $x\in P_i\subset W^n$, we call $\tau_i(x)$
           the twin point $x$ in $P_i$. Moreover, if $x$ is in the relative interior of
           $P_{i_1}\cap \cdots \cap P_{i_s}$, any
            $\tau^{\varepsilon_s}_{i_s} \circ \cdots \circ \tau^{\varepsilon_1}_{i_1}
            (x)$ where $\varepsilon_i \in \{ 0 ,1 \}$ is called a duplicate point of $x$ in $W^n$.
            But in this case, it is not necessarily that $x$ has exactly $2^s$
            duplicate points (even if each $\tau_i$ on $P_i$ is free, see Figure~\ref{p:CountExample}).
            Also we can define subpanels for $W^n$ as in~\eqref{Exam:subpanel}.

        \vskip .2cm

         \begin{rem}
         A nice manifold with corners $W^n$ may admit many different involutive
         panel structures on the boundary (for example see
         Figure~\ref{p:three-Panel-Structures}).
         \begin{figure}
         \includegraphics[width=0.45\textwidth]{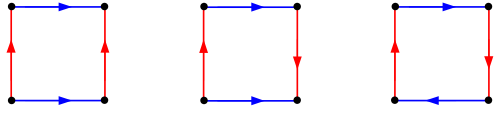}
          \caption{Three different involutive
         panel structures on a square  }\label{p:three-Panel-Structures}
      \end{figure}
      \end{rem}

       \begin{exam} \label{Exam:Trivial-Panel}
            Suppose $X^n$ is a nice manifold with corners, and let
            $F_1, \cdots, F_l$ be all the facets of $X^n$.
           We can think of
           $X^n$ having a \textit{trivial involutive panel structure}
           which is defined by: for any $1\leq i \leq l$,
           $P_i =F_i$ and the involution $\tau_i = id_{F_i} : F_i \rightarrow F_i$.
           In this case, we call each $P_i$ a \textit{reflexive panel} of $X^n$.
           Obviously, in the trivial involutive panel structure,
           any point of $X^n$ has only one duplicate
           point --- itself.
       \end{exam}   \vskip .2cm

       In general, suppose $\{ P_i, \tau_i \}_{1\leq i \leq k}$ is an involutive
       panel structure on a nice manifold with corners $W^n$. If
       the $\tau_i : P_i \rightarrow P_i$ is the identity map, we
       call $P_i$ a \textit{reflexive panel} of $W^n$.
       \\

       \begin{exam} \label{Exam:subpanel}
         Suppose $V^n$ is a $\Z_2$-core of $Q^n$. Use the notations
         above, for any panel $P_i$ of $V^n$, $P_i$ itself is a
         nice manifold with corners and has an involutive
         panel structure on the boundary induced from $V^n$. Indeed, the panels of
         $P_i$ are:
            $$\{ P_j\cap P_i ; \, \tau_j|_{ P_j\cap P_i} :  P_j\cap P_i \rightarrow  P_j\cap P_i \
            \text{for}\ \forall\,  1\leq j \leq k, j\neq i \}.$$
         More generally, for an $I= \{ i_1, \cdots, i_s \} \subset \{ 1,\cdots, k \}$,
        the subpanel $ P_{I}$ is
         an $(n-s)$-dimensional nice
         manifold with corners (may not be connected), and $P_I$ has an involutive
         panel structure on its boundary induced from $V^n$
         which is given by:
         $$\{ P_j\cap P_I \neq \varnothing ;\,
           \tau_j|_{ P_j\cap P_I} :  P_j\cap P_I \rightarrow  P_j\cap P_I \
            \text{for}\ \forall\, 1\leq j \leq k, j\notin I
         \}.$$
       \end{exam}
      \vskip .2cm

        Obviously, the $\Z_2$-core of a closed manifold $Q^n$ is far from
         unique. The topological type of a $\Z_2$-core depends on the
             corresponding $\Z_2$-cut system in $Q^n$. For an arbitrary
             $\Z_2$-cut system of $Q^n$, it is fairly possible that
             the corresponding $\Z_2$-core of $Q^n$ is not connected.
             But we can prove the following statement
             (which is not used in any other place in the paper).\vskip .4cm

     \begin{thm} \label{Thm:Conn_Z2-core}
      For any closed connected manifold $Q^n$, there always exists
      a connected $\Z_2$-core for $Q^n$.
     \end{thm}
     \begin{proof}
         For $n=1$ and $2$, the statement is obviously true. So
          assume $n\geq 3$ in the rest of the proof.
         \vskip .2cm

        First, let us choose a $\Z_2$-cut system $\Sigma_1,\cdots, \Sigma_k$ of $Q^n$
        with each $\Sigma_i$ being connected.
        We claim that each $\Sigma_i$ is non-separating in
          $Q^n$. let $\{ [\Gamma_1], \cdots, [\Gamma_k] \}\subset
        H_1(Q^n, \Z_2)$ the dual basis of the
         $\{ [\Sigma_1], \cdots ,[\Sigma_k] \} \subset H_{n-1}(Q^n, \Z_2)$ under
        $\Z_2$-intersection form of $Q^n$, i.e.
         \[ \# (\Gamma_i \cap \Sigma_j) = \delta_{ij} \mod 2
         \]
           So the curve $\Gamma_i$ must intersect
          $ \Sigma_i$ odd number of times. Let $N(\Sigma_i)$ be a small tubular neighborhood
          of $\Sigma_i$ in $Q^n$. Since $\Sigma_i$ is connected,
          $Q^n- int(N(\Sigma_i))$ is either connected or has exactly two connected-component.
          but the later case contradicts $\#(\Gamma_i \cap \Sigma_i ) =1 \mod 2$.
           So $\Sigma_i$ must be non-separating in $Q^n$.\vskip .2cm

           Let $Q^n_j$ be the manifold we get by
         cutting $Q^n$ open along $\{ \Sigma_1,\cdots,\Sigma_j \}$, i.e.
           $$ Q^n_j = Q^n - \bigcup^{j}_{i=1} int(N(\Sigma_i)) .$$
         In addition, for any $j+1 \leq i \leq k $,
         let $\Sigma^{(j)}_i := \Sigma_i\cap Q^n_j$ and $\Gamma^{(j)}_i := \Gamma_i\cap
         Q^n_j$.\vskip .2cm

            Assume $Q^n_j$ is connected and we
          cut $Q^n_j$ open along $\Sigma^{(j)}_{j+1}$.
         Since the relative
         intersection number of $\Sigma^{(j)}_{j+1}$ and $\Gamma^{(j)}_{j+1}$ in
         $H_*(Q^n_j, \partial Q^n_j,\Z_2)$ is $1$ ($\mathrm{mod}\; 2$), if
          $\Sigma^{(j)}_{j+1}$ is connected in $Q^n_j$,  then
          $\Sigma^{(j)}_{j+1}$ must be non-separating in $Q^n_j$ for the
          same reason as above. If $\Sigma^{(j)}_{j+1}$ is
          not connected in $Q^n_j$, we can connect all the components of
          $\Sigma^{(j)}_{j+1}$ via
           some thin tubes in $Q^n_j$ which are transverse to
           other $\Sigma^{(j)}_{i}$'s.
           This operation will change
           the original $\Sigma_{j+1}$ in $Q^n$ simultaneously,
           but it will not change the homology class
           of $\Sigma_{j+1}$ in $H_{n-1}(Q^n,\Z_2)$. Now, cutting $Q^n_j$ open along
           the new $\Sigma^{(j)}_{j+1}$,
           we get a nice manifold with corners $Q^n_{j+1}$ which remains connected.
           \vskip.2cm

            By iterating the above argument from $j=1$ to $j=k$, we will get a connected
           nice manifold with corners $V^n$. By definition, $V^n$ is the
           $\Z_2$-core of $Q^n$ from cutting $Q^n$ open along a $\Z_2$-cut
           system $\{\Sigma'_1,\cdots, \Sigma'_k \}$, which is obtained from the original
           $\Z_2$-cut system by some homology preserving operations.
     \end{proof}

      \vskip .6cm

   \section{Construction of free $(\Z_2)^m$-actions on closed manifolds} \label{Sec3}

      Suppose $\pi : M^n \rightarrow Q^n$
       is a principal
       $(\Z_2)^m$-bundle over a closed connected manifold $Q^n$.
       Let $V^n$ be a $\Z_2$-core of $Q^n$ from cutting $Q^n$
       along a $\Z_2$-cut system
       $\{ \Sigma_1, \cdots ,\Sigma_k \} $ in $Q^n$ (we do
       not assume that $V^n$ is connected).
        We have shown that the principal $(\Z_2)^m$-bundle $\pi$ is classified by an element
            \begin{equation} \label{Equ:Classify-Bundle}
          \Lambda_{\pi} \in H^1(Q^n,(\Z_2)^m) \cong \mathrm{Hom}(H_1(Q^n,\Z_2),(\Z_2)^m)
            \end{equation}

        By the Poincar\'e duality, there is an isomorphism $\psi: H_{n-1}(Q^n,\Z_2) \cong H_1(Q^n,\Z_2)$.
        So we get an element $\Lambda^*_{\pi} \in
        \mathrm{Hom}(H_{n-1}(Q^n,\Z_2),(\Z_2)^m)$ defined by:
           \begin{align*}
             \Lambda^*_{\pi}: \ \{ [\Sigma_1], \cdots ,[\Sigma_k] \}
             &\longrightarrow (\Z_2)^m \\
            [\Sigma_i] \quad\, &\mapsto\;\, \Lambda_{\pi}(\psi([\Sigma_i]))
          \end{align*}
         Let $P_i\subset \partial V^n$
         be the panel corresponding to $\Sigma_i$. So we have a map
          \begin{align*}
             \lambda_{\pi}: \ \{ P_1, \cdots , P_k \}
             &\longrightarrow (\Z_2)^m \\
              P_i \quad\; &\mapsto\; \Lambda^*_{\pi}([\Sigma_i])=\Lambda_{\pi}(\psi([\Sigma_i]))
          \end{align*}
          $\lambda_{\pi}$ is called the
        \textit{associated $(\Z_2)^m$-coloring}
         of $\pi: M^n \rightarrow Q^n$ on $V^n$.
          In general, any map $\lambda: \{ P_1, \cdots, P_k \} \rightarrow
          (\Z_2)^m$ is called a \textit{$(\Z_2)^m$-coloring on
          $V^n$}, and any element in $(\Z_2)^m$ is called a \textit{color}.\vskip .2cm

         In addition, if we consider $\pi: M^n \rightarrow Q^n$
         as a regular covering, the $\Lambda_{\pi}$ is just
         the abelianization of the monodromy map $\mathcal{H}_{\pi} :
         \pi_1(Q^n, q_0) \rightarrow (\Z_2)^m$,
          where $q_0\in Q^n$ is a base point and $(\Z_2)^m$ is identified with
          the deck transformation group of this covering $\pi$. Indeed, let
           $ \{ [\Gamma_1], \cdots, [\Gamma_k] \} \subset H_1(Q^n, \Z_2) $
        be the dual basis of $ \{ [\Sigma_1], \cdots ,[\Sigma_k] \} $
        under the $\Z_2$-intersection form of $Q^n$ where each $\Gamma_i$ is
        a closed curve that intersects all $\Sigma_j$'s transversely.
         If we fix a point $x_0\in \pi^{-1}(q_0)$, and let $\widetilde{\Gamma}_i
         : [0,1] \rightarrow M^n$ be a lifting of $\Gamma_i$ with
         $\widetilde{\Gamma}_i(0)=x_0$, then
        \begin{equation} \label{Equ:Classify-Monodromy}
          \widetilde{\Gamma}_i(1) = \mathcal{H}_{\pi}(\Gamma_i)\cdot
        x_0= \Lambda_{\pi}([\Gamma_i])\cdot x_0.
         \end{equation}

       \vskip .2cm

        Conversely, given an arbitrary $(\Z_2)^m$-coloring
           $\lambda$ on $V^n$, we can construct a principal $(\Z_2)^m$-bundle
        over $Q^n$ by the following rule:
         \begin{equation} \label{Glue_Back}
            M(V^n, \{ P_i, \tau_i \},\lambda) := V^n \times (\Z_2)^m \slash \sim
         \end{equation}
           Where $(x,g)\sim (x',g') $ whenever
            $x'= \tau_i(x)$ for some $P_i$ and
              $g' = g+ \lambda(P_i) \in (\Z_2)^m$. \vskip .2cm
          It is easy to see that
             if $x$ is in the interior of $P_{i_1} \cap \cdots \cap
             P_{i_s}$,
           $(x,g) \sim (x',g')$ if and only if
           $ (x',g')= ( \tau^{\varepsilon_s}_{i_s}\circ \cdots \circ
           \tau^{\varepsilon_1}_{i_1}(x), g + \varepsilon_1\lambda(P_1) + \cdots +
           \varepsilon_s\lambda(P_s))$
           where $\varepsilon_j \in \{0, 1\}$ for $\forall\, 1\leq  j \leq s$. \vskip .2cm

           We call $M(V^n,\{ P_i, \tau_i \},\lambda)$ the
          \textit{glue-back construction} from $(V^n,\lambda)$.
            Also, we use $M(V^n,\lambda)$
            to denote $M(V^n,\{ P_i, \tau_i \},\lambda)$
            if there is no ambivalence about the involutive panel structure on $V^n$
            in the context.
            \vskip .2cm

           Let $[(x,g)]\in M(V^n,\lambda)$ denote the equivalence class
           of $(x,g)$ defined in~\eqref{Glue_Back}.
           Then we can define a natural $(\Z_2)^m$-action on $M(V^n,\lambda)$
            by:
           \begin{equation} \label{Equ:FreeAction}
              g\cdot [(x,g_0)] := [(x, g+g_0)],\; \forall\, x\in V^n, \
             \forall\, g, g_0\in (\Z_2)^m.
           \end{equation}

         It is easy to check that the $(\Z_2)^m$-action
          is well defined. And for any element $g\neq 0 \in (\Z_2)^m$,
         $ g\cdot [(x,g_0)] = [(x,g+g_0) ]\neq [(x, g_0)]$. This is
         because: \vskip .2cm
         \begin{itemize}
           \item[(i)] when $x$ is in the interior of $V^n$, $(x,g+g_0)$ and
         $(x,g_0)$ are not equivalent under $\sim$ for any $g \neq
         0$; \vskip .1cm

           \item[(ii)] when $x$ is in the relative interior of $P_{i_1} \cap \cdots \cap
           P_{i_s}$,
           $(x,g+g_0) \sim (x,g_0)$ would force
           $(x,g+g_0) =( \tau^{\varepsilon_s}_{i_s}\circ \cdots \circ
           \tau^{\varepsilon_1}_{i_1}(x), g_0 + \varepsilon_1\lambda(P_1) + \cdots +
           \varepsilon_s\lambda(P_s))$.
            Notice that $g\neq 0$ implies that at least one of the
           $\varepsilon_1, \cdots, \varepsilon_s$ is not $0$. But
           since $x$ has exactly $2^s$ duplicate points in $V^n$,
            $\tau^{\varepsilon_s}_{i_s}\circ \cdots \circ\tau^{\varepsilon_1}_{i_1}(x) \neq x$
           as long as some $\varepsilon_j \neq 0$.
         \end{itemize}
         \vskip .1cm
            So the action of $(\Z_2)^m$ on $M(V^n,\lambda)$ defined
            by~\eqref{Equ:FreeAction} is always a free group action.
            In the rest of this paper, we will always assume that $M(V^n,\lambda)$ is
            equipped with this free $(\Z_2)^m$-action.

         \vskip .2cm

        \begin{rem}
          Since $Q^n$ is smooth,
          the $M(V^n,\lambda)$ is naturally a smooth manifold and
       the natural $(\Z_2)^m$-action on $M(V^n,\lambda)$ defined in~\eqref{Equ:FreeAction}
       is smooth. \vskip .2cm
      \end{rem}

      \begin{rem}
          A similar idea to the glue-back construction was
          used to construct cyclic and infinite cyclic covering spaces of
            the complement of knots in $S^3$ (see~\cite{Rolfsen76}).
       \end{rem} \vskip .2cm

          \begin{thm} \label{thm:manifold}
             $M(V^n,\lambda)$ is a closed $n$-manifold and the orbit
             space of
          the free $(\Z_2)^m$-action on $M(V^n,\lambda)$ defined in~\eqref{Equ:FreeAction}
          is homeomorphic to $Q^n$.
          \end{thm}
         \begin{proof} Observe that each orbit of this $(\Z_2)^m$-action
          has some representative in $V^n\times \{ 0 \}$. And for any
          point $x$ in the relative interior of $P_{i_1} \cap \cdots \cap P_{i_s}$, a point $(x',0) \in
          V^n\times \{ 0 \}$ is in the same orbit as $(x,0)$ under the
          above $(\Z_2)^m$-action if and only if $x'= \tau^{\varepsilon_s}_{i_s}\circ \cdots \circ
           \tau^{\varepsilon_1}_{i_1}(x) $ for some $\varepsilon_{1},\cdots
          \varepsilon_{s} \in \{ 0,1 \}$ (in other words, $x'$ is a duplicate point
           of $x$ in $V^n$).
          So the orbit space is homeomorphic to the space of gluing all
          points of $V^n$ with their duplicate points together, which is
          homeomorphic to $Q^n$ by Lemma~\ref{Lem:Quotient-Homeo}.
          And since $Q^n$ is a closed manifold, so is $M(V^n,\lambda)$.
         \end{proof}
         \vskip .2cm

       Following are some explicit examples of free $(\Z_2)^m$-actions on manifolds
        from the glue-back construction.\vskip .2cm

   \begin{exam}
         A meridian and a longitude of the torus $T^2$ forms a
         $\Z_2$-cut system of $T^2$. The corresponding $\Z_2$-core
         of $T^2$ is a
         square $V^2$. Given a coloring of the edges of $V^2$ by
         elements in $(\Z_2)^2= \langle e_1\rangle \oplus \langle e_2\rangle$
          such that opposite edges of $V^2$ are colored by
         the same element of $(\Z_2)^2$, we can construct all principal
         $(\Z_2)^2$-bundles over $T^2$ (see Figure~\ref{p:Torus_1}
         and Figure~\ref{p:Torus_2} for such examples).
     \end{exam}

          \begin{figure}
          \includegraphics[width=0.72\textwidth]{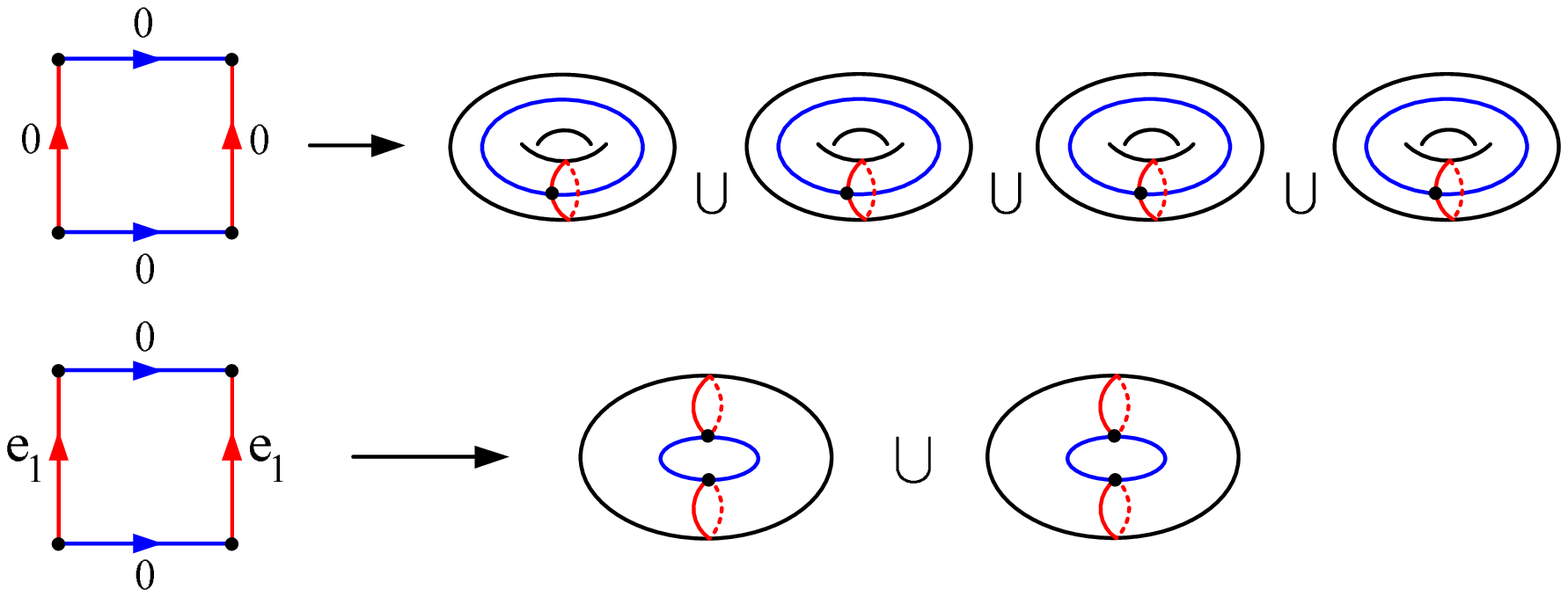}
            \caption{Examples of the glue-back construction}\label{p:Torus_1}
          \end{figure}

             \begin{figure}
          \includegraphics[width=0.53\textwidth]{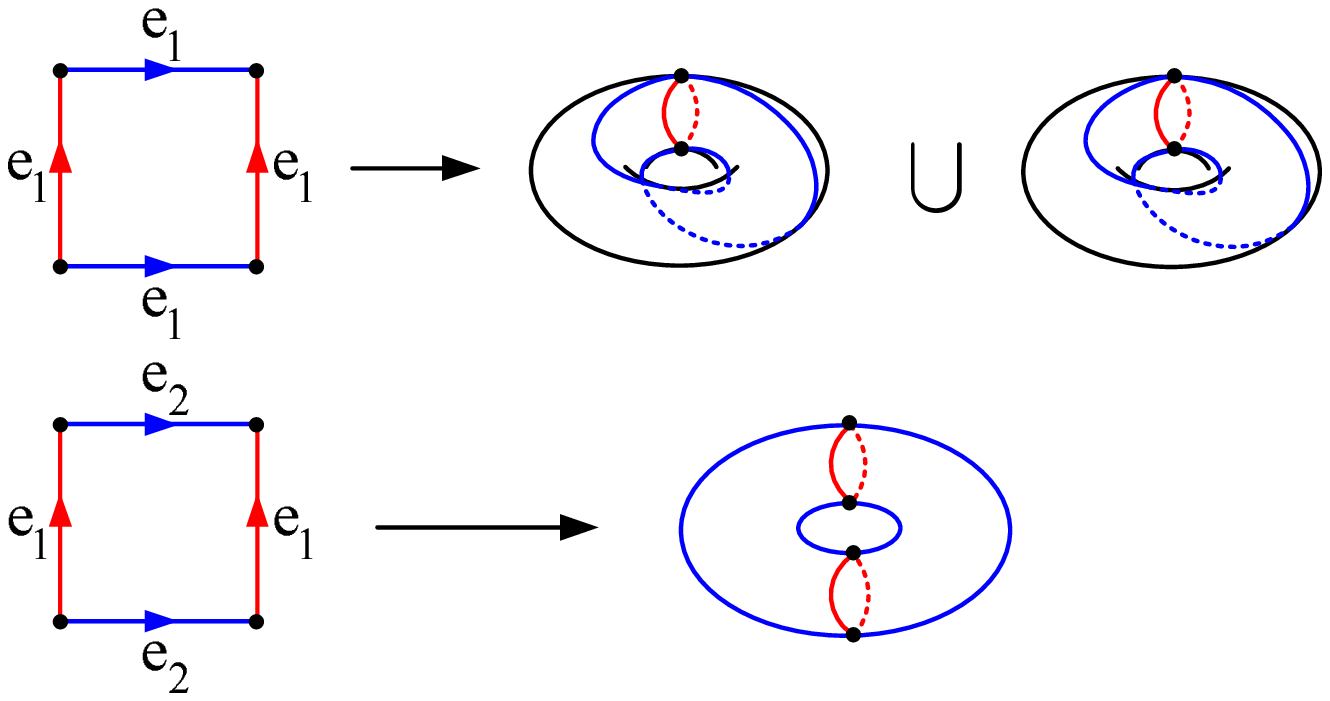}
            \caption{ Examples of the glue-back construction
            }\label{p:Torus_2}
          \end{figure}

     \begin{exam}
           Let $M^2$ be a disjoint union of two
           $S^2$. Figure~\ref{p:Spheres} shows a free $(\Z_2)^2$-action
           on $M^2$ whose orbit space is $\R P^2$. A $\Z_2$-core
              $V^2$ of $\R P^2$ is a disk with only one panel $P
             =\partial V^2$. Let $\{ e_1, e_2\}$ be a basis of $(\Z_2)^2$.
              Then $M^2\cong M(V^2, \lambda)$ where $\lambda$ is a $(\Z_2)^2$-coloring on
              $V^2$ given by $\lambda(P)=e_1$
              (or $e_2$).

              \begin{figure}
              \includegraphics[width=0.55\textwidth]{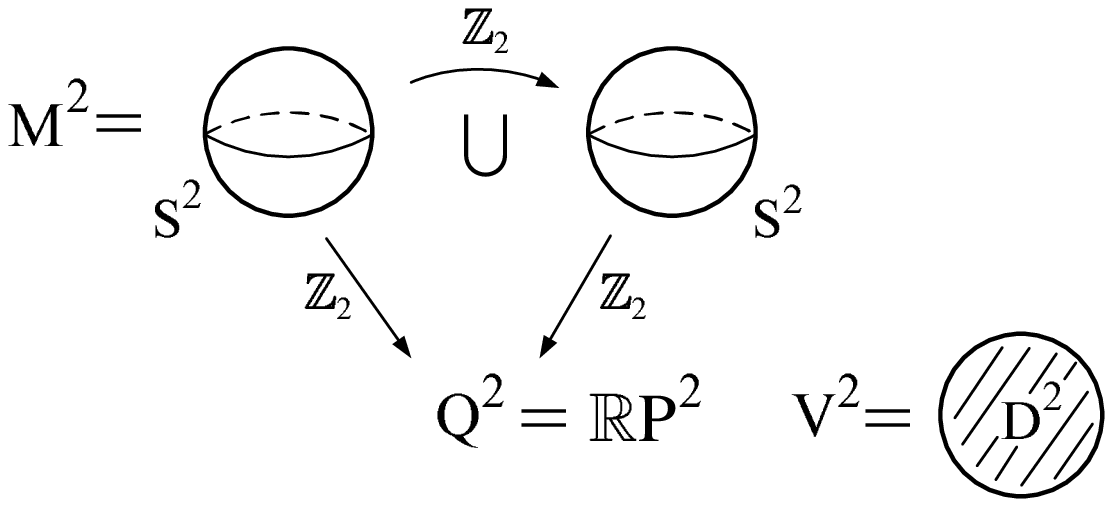}
              \caption{ }\label{p:Spheres}
            \end{figure}
    \end{exam}
     \vskip .2cm

          More generally, for any nice manifold with corners $W^n$ with an
          involutive panel structure $\{ \tau_i : P_i \rightarrow P_i \}_{1\leq i \leq k}$,
           any map $\mu: \{ P_1, \cdots, P_k  \} \rightarrow (\Z_2)^m $ is called
            a $(\Z_2)^m$-coloring on $W^n$. We can define the glue-back
           construction $M(W^n,\mu)$ by the same rule as in~\eqref{Glue_Back}. Also we have
           a natural $(\Z_2)^m$-action on $M(W^n,\mu)$ defined
           by~\eqref{Equ:FreeAction}. But this $(\Z_2)^m$-action
            on $M(W^n,\mu)$ may not be free. Indeed, suppose
            $\theta_{\mu}: W^n\times (\Z_2)^m \rightarrow M(W^n,\mu) $
           is the quotient map. For a point $x$ in the relative interior
           of a codimension $s$ face of $W^n$, it is possible that $x$ has less than
           $2^s$ duplicate points in $W^n$. In that case, $\theta_{\mu}(x\times (\Z_2)^m)$
            would consist of less than $2^m$ points, which implies that $\theta_{\mu}(x\times (\Z_2)^m)$
            can not a free orbit under the natural $(\Z_2)^m$-action
            on $M(W^n,\mu)$ defined in~\eqref{Equ:FreeAction} (see the examples below).
            \vskip .2cm

           \begin{exam} \label{Exam:small-cover}
              For a simple polytope $V^n$, consider each facet of $V^n$ as a panel and
              $V^n$ has the trivial involutive panel structure (see Example~\ref{Exam:Trivial-Panel}).
              Then a small
           cover over $V^n$ can be thought of as the glue-back construction $M(V^n,\lambda)$
           where $\lambda$ is a \textit{characteristic function} on $V^n$ with value
           in $(\Z_2)^n$ (see~\cite{DaJan91}). But the natural action of $(\Z_2)^n$ defined
           by~\eqref{Equ:FreeAction} on a small cover is exactly the locally standard
           action defined in~\cite{DaJan91}, which is not free.\vskip .2cm
           \end{exam}

           \begin{rem}
            From the Example~\ref{Exam:small-cover}, we can see that the significance of introducing
            the general notion of involutive panel structure in Definition~\ref{Def:Invol_Panel}
            is that:
            it allows us to unify the way of constructing free $(\Z_2)^m$-actions and
         non-free locally standard $(\Z_2)^m$-actions on manifolds from the orbit
         spaces (see Section~\ref{Sec5} for details).
           \end{rem}

           \begin{exam} \label{Exam:Branch-Cover}
             Suppose a square $[0,1]^2$ is equipped with an
             involutive panel structure as indicated by the arrows in
             Figure~\ref{p:CountExample}.
            For a $(\Z_2)^2$-coloring $\lambda$ defined by $\lambda(P_1)=e_1$,
             $\lambda(P_2)=e_2$ where $\{ e_1, e_2 \}$ is a basis of
             $(\Z_2)^2$,
             the glue-back construction $M([0,1]^2,\lambda)$ is homeomorphic to $T^2$.
             But the natural $(\Z_2)^2$-action on $T^2$ defined by~\eqref{Equ:FreeAction} is not free.
           \end{exam}
             \vskip .2cm

       \begin{figure}
         \includegraphics[width=0.41\textwidth]{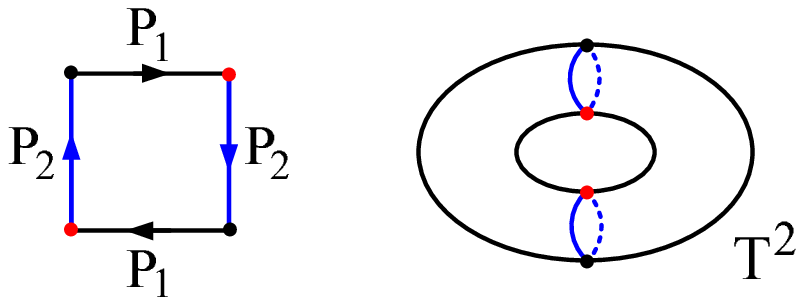}
          \caption{}\label{p:CountExample}
      \end{figure}

          For a nice manifold with corners $W^n$ equipped with an
          involutive panel structure, if for any $s\geq 0$,
          any point in the relative interior
           of any codimension $s$ face of $W^n$ has exactly
           $2^s$ duplicate points in $W^n$, the involutive panel structure
           is called \textit{perfect}. For example: the involutive panel structure
           on any $\Z_2$-core $V^n$ of $Q^n$ constructed from
           Lemma~\ref{Lem:Deform_Invol} above is perfect.\vskip .2cm

            We can easily show that if the involutive panel structure on $W^n$
           is perfect,
          the natural $(\Z_2)^m$-action on $M(W^n,\lambda)$ defined by~\eqref{Equ:FreeAction}
          is free for all $(\Z_2)^m$-coloring $\lambda$ on $W^n$.
           However, even if the involutive panel structure on $W^n$ is
           not perfect, it is still possible that there exists some
           nontrivial $(\Z_2)^m$-coloring $\lambda$ on $W^n$ so that
            the natural $(\Z_2)^m$-action on $M(W^n,\lambda)$ is
            free. For example, although the involutive panel structure on
            the square in Example~\ref{Exam:Branch-Cover} is not perfect,
           if we color the two panels
            of the square both by $e_1\in (\Z_2)^2$, we will obtain a disjoint union of
            two spheres $S^2\cup S^2$ from the glue-back construction. Obviously,
             the natural $(\Z_2)^2$-action on this $S^2\cup S^2$
             is free. \vskip .2cm

        Let $V^n$ be a $\Z_2$-core of a closed manifold $Q^n$ described as above.
        As in Example~\ref{Exam:subpanel}, we can think of a panel $P_i \subset V^n$
        itself as a nice manifold with corners with an involutive
         panel structure defined by
          $\{ P_j\cap P_i ;\, 1\leq j \leq k, j\neq i \}$.
         Then we have an induced $(\Z_2)^m$-coloring $\lambda^{in}_{P_i}$ of $P_i$
         defined by
         \[  \lambda^{in}_{P_i} (P_j \cap P_i) :=
           \lambda(P_j) , \ \quad \forall\,  j\neq i,  P_j \cap P_i\neq \varnothing
         \]

         Furthermore, for $I=\{ i_1, \cdots, i_s\} \subset \{ 1,\cdots, k \}$,
         the subpanel
         $ P_{I} = P_{i_1} \cap \cdots \cap P_{i_s}$ has an involutive
         panel structure on its boundary defined by
          $\{ P_j\cap P_I \neq \varnothing ;\, 1\leq j \leq k, j\notin I
         \}$.  The \textit{induced $(\Z_2)^m$-coloring} $\lambda^{in}_{P_I}$
          of $P_I$ is
         \begin{equation} \label{Equ:color}
           \lambda^{in}_{P_I} (P_j \cap P_I) = \lambda(P_j) \in
           (\Z_2)^m
           ,\quad 1\leq j \leq k,\, j\notin I,\, P_j \cap P_I\neq \varnothing
           \end{equation}

         If we apply the glue-back construction~\eqref{Glue_Back} to $(P_I,
          \lambda^{in}_{P_I})$, we will get a closed manifold $M(P_I,\lambda^{in}_{P_I})$.
      Notice that when $|I|\geq 1$, by the
          definition of $M(P_I,\lambda^{in}_{P_I})$, the relative interior
          points of the $2^m$ copies of $P_I$ are not glued together like they are in
          $M(V^n,\lambda)$. In fact, it is easy to see that
          $M(P_I,\lambda^{in}_{P_I})$ is homeomorphic to
          a disjoint union of $2^s$ copies of
          $\theta_{\lambda}^{-1}(\Sigma_I)$, where
          $\theta_{\lambda}: V^n \times (\Z_2)^m \rightarrow M(V^n,\lambda)$
          is the quotient map defined in~\eqref{Glue_Back}.

     \vskip .2cm

     \begin{thm} \label{Thm:Equiv-Isom}
         For any principal $(\Z_2)^m$-bundle
         $\pi: M^n \rightarrow Q^n$, let $\lambda$
         be the associated $(\Z_2)^m$-coloring of $\pi$ on
         $V^n$. Then there is an equivariant homeomorphism
          from $M(V^n, \lambda)$ to $M^n$ which covers the identity of
          $Q^n$.
     \end{thm}

        \begin{proof}
          Let $\theta_{\lambda}: V^n\times (\Z_2)^m \rightarrow M(V^n,\lambda)$ be the
         quotient map and
         $\xi_{\lambda} : M(V^n,\lambda) \rightarrow Q^n $ be
         the orbit map of the natural $(\Z_2)^m$-action defined by~\eqref{Equ:FreeAction}.
        It suffice to show that $\xi_{\lambda}$ and $\pi$ defines the same
         monodromy map as regular coverings over $Q^n$.
       So let us first compute the monodromy $ \mathcal{H}_{\xi_{\lambda}} ([\Gamma])$ for
     any closed curve $\Gamma: [0,1] \rightarrow Q^n$. \vskip .2cm

      Suppose $\Gamma(t)$ meets $\Sigma_{i_1}, \cdots, \Sigma_{i_r}$
      consecutively in $Q^n$ as the time $t$ goes from $0$ to $1$.
     When cutting $Q^n$ open
      along
      $\{ \Sigma_1, \cdots, \Sigma_k \}$, the cut-open image of
      $\Gamma$ is $ \Gamma \cap V^n :=\gamma$.
      Note that the curve $\gamma$ might
       be disconnected in $V^n$.
      When the time parameter increases,
       $\gamma$ will meet the panels $P_{i_1}, \cdots, P_{i_r}$ of $V^n$ consecutively.
       Then when we glue the $2^m$ copies of $V^n$ together in
       the glue-back construction, the $\gamma$ in different copies of $V^n$ are fit
        together which gives all the liftings of $\Gamma$ in $M(V^n,\lambda)$.
       Indeed, if we choose the start point of a lifting of $\Gamma$
       in $\theta_{\lambda}(V^n \times g_0)$, $g_0\in (\Z_2)^m$, the end point of
       the lifting would be in $\theta_{\lambda} \left(V^n \times
        (g_0 + \lambda(P_{i_1}) + \cdots + \lambda(P_{i_r})) \right)$.
         So the monodromy $\mathcal{H}_{\xi_{\lambda}}$ of $\xi_{\lambda}$ is:
      \begin{equation} \label{Equ:Monodromy}
       \mathcal{H}_{\xi_{\lambda}}([\Gamma])= \lambda(P_{i_1}) + \cdots + \lambda(P_{i_r})
           \in (\Z_2)^m.
       \end{equation}

        Now let $\{ [\Gamma_1], \cdots, [\Gamma_k] \}\subset
        H_1(Q^n, \Z_2)$ the dual basis of the
         $\{ [\Sigma_1], \cdots ,[\Sigma_k] \} \subset H_{n-1}(Q^n, \Z_2)$ under
        $\Z_2$-intersection form of $Q^n$. Then
         \begin{equation} \label{Equ:Intersect_Num}
           \# (\Gamma_i \cap\Sigma_j) = \delta_{ij} \mod 2
         \end{equation}
        We can assume that each $\Gamma_i : [0,1] \rightarrow Q^n$
        is a closed curve which intersects all $\Sigma_j$'s transversely
        and starts at the same base point $q_0\in Q^n$.
        Suppose $\gamma_i = \Gamma_i\cap V^n$ is the cut-open image
         of $\Gamma_i$ in $V^n$. Then by~\eqref{Equ:Intersect_Num}, $\gamma_i$ will
         meet $P_i$ odd number of times and meet all other
         $P_j$ ($j\neq i$) even number of times.
       So by~\eqref{Equ:Monodromy}, we have:
         \begin{equation} \label{Equ:Monodromy-Equal}
          \mathcal{H}_{\xi_{\lambda}}([\Gamma_i]) = \lambda(P_i) =
         \mathcal{H}_{\pi}([\Gamma_i]),\  1\leq i \leq k.
         \end{equation}
         This implies that $\mathcal{H}_{\xi_{\lambda}} =
       \mathcal{H}_{\pi}$. So the theorem is prove.
      \end{proof}
          \vskip .2cm

   \begin{rem}
       For a $\Z_2$-core $V^n$ of $Q^n$ with $H^1(V^n,\Z_2) \neq 0$,
       a principal $(\Z_2)^m$-bundle over $V^n$ is not necessarily
       trivial. If we apply the gluing rule~\eqref{Glue_Back} to
       an arbitrary principal $(\Z_2)^m$-bundle over $V^n$, we may get a principal
       $(\Z_2)^m$-bundle over $Q^n$ too. The significance of
        Theorem~\ref{Thm:Equiv-Isom} is that we can
       actually use
        the trivial $(\Z_2)^m$-bundle over $V^n$ (i.e. $V^n \times
       (\Z_2)^m$) and the gluing rule~\eqref{Glue_Back} to construct
       all principal $(\Z_2)^m$-bundles over
       $Q^n$, which is enough for our purpose in this paper.
   \end{rem}
     \vskip .2cm

     For a $(\Z_2)^m$-coloring $\lambda$ on the panels
         $ P_1,\cdots, P_k $ of $V^n$, define:
     \begin{align}
       L_{\lambda} &:= \text{the subgroup of $(\Z_2)^m$ generated
         by $ \{ \lambda(P_1), \cdots ,\lambda(P_k) \} $},   \label{Equ:L-lambda} \\
          & \text{rank}(\lambda) :=  \mathrm{dim}_{\Z_2} \label{Equ:rank-lambda}
          L_{\lambda}.
     \end{align}

    \vskip .2cm

     Since $\lambda$ encodes all the structural information of
     $M(V^n,\lambda)$, so any topological invariant of $M(V^n,\lambda)$
       (e.g. homology groups) should be completely determined by $(V^n,\lambda)$. But
       if we try to compute the $Z_2$-homology groups of $M(V^n,\lambda)$
       via the Serre-spectral sequence,
     the problem of twisted local coefficients could occur when the orbit space is not
      simply-connected. This problem is hard to get around in general.
     However, we can at least compute $H_0(M(V^n,\lambda),\Z_2)$,
     i.e. the number of connected components of
     $M(V^n,\lambda)$, from the $(\Z_2)^m$-coloring $\lambda$. \vskip .3cm

       \begin{thm} \label{thm:comp}
        For any $(\Z_2)^m$-coloring $\lambda$ of $V^n$,
     $M(V^n,\lambda)$ has $2^{m-\mathrm{rank}(\lambda)}$ connected
       components which are pairwise homeomorphic. Let $\theta_{\lambda}:
       V^n\times (\Z_2)^m \rightarrow M(V^n,\lambda)$ be the
       quotient map. Then each
       connected component of $M(V^n,\lambda)$ is homeomorphic to
       $\theta_{\lambda}(V^n\times L_{\lambda})$.
        And there is a free action of $L_{\lambda} \cong (\Z_2)^{\mathrm{rank}(\lambda)}$ on
       each connected component of $M(V^n,\lambda)$ whose orbit space is $Q^n$.
      \end{thm}
      \begin{proof}
       Let $\theta_{\lambda}: V^n \times (\Z_2)^m \rightarrow M(V^n,\lambda)$
       be the quotient map defined by~\eqref{Glue_Back}. Here we do not assume $V^n$ is connected.
       Then by the definition of
       $M(V^n,\lambda)$, for two arbitrary connected components $K, K'$ of $V^n$
        and $\forall\, g,g'\in (\Z_2)^m$, we have:
    \begin{enumerate}
         \item[(a)] $\theta_{\lambda}(K\times g)$ and $\theta_{\lambda}(K'\times g')$
           are in the same connected component of $M(V^n,\lambda)$ if and only if there
           is a sequence
           $$ \quad\qquad (K,g) = (K_0, g_0) \leftrightarrow (K_1, g_1)
            \cdots \leftrightarrow (K_{r-1}, g_{r-1})
           \leftrightarrow (K_r, g_r) = (K',g')
            $$
            where each $K_i$ is a connected component of $V^n$, $g_i\in
            (\Z_2)^m$, and
            $\theta_{\lambda}(K_i\times g_i)$ and $\theta_{\lambda}(K_{i+1}\times g_{i+1})$
            share an $(n-1)$-dimensional face in $M(V^n,\lambda)$.  \vskip .2cm

         \item[(b)]  $\theta_{\lambda}(K\times g)$ and $\theta_{\lambda}(K'\times g')$
        share an $(n-1)$-dimensional face in $M(V^n,\lambda)$ if and only
        if there is a facet $F$ of $K$ with its twin facet $F^* \subset K'$ and
        $g' - g  = \lambda(F) $.  \vskip .2cm
    \end{enumerate}
    So if $\theta_{\lambda}(K\times g)$ and $\theta_{\lambda}(K'\times g')$
    are in the same connected component of $M(V^n,\lambda)$,
    it is necessary that
    $g' \in g + L_{\lambda} $. \vskip .2cm

       Conversely,
       we claim: for any $g' \in g + L_{\lambda}$,
         $\theta_{\lambda}(K\times g)$ and $\theta_{\lambda}(K'\times g')$
        are always in the same connected component of $M(V^n,\lambda)$ for
        any connected components $K$ and $K'$ of $V^n$. \vskip .2cm

        Indeed, since $Q^n$ is connected,
        for $K$ and $K'$, there always exists a sequence,
        $K=K_0, K_1, \cdots, K_{r-1}, K_r = K'$,
         such that some facet $F_{a_i} \subset K_i$ while $F^*_{a_i} \subset
         K_{i+1}$. So by the above argument,
         $\theta_{\lambda}(K\times g)$ lies in the same connected component
         as $\theta_{\lambda}(K'\times g^*)$ in $M(V^n,\lambda)$ for some $g^* \in
         g + L_{\lambda}$.
        Then it remains to show that $\theta_{\lambda}(K'\times g^*)$ and
        $\theta_{\lambda}(K'\times g')$
    are always in the same connected component of $M(V^n,\lambda)$
    whenever $g'-g^* \in L_{\lambda}$. \vskip .2cm

        To see this, let $\Gamma: [0,1]\rightarrow Q^n $ be an arbitrary closed curves
          based at a point $q_0 \in K' \subset V^n $.
         For any $g^*\in (\Z_2)^m$, there is a lifting of $\Gamma$ in $M(V^n,\lambda)$
         which goes from a point in $\theta_{\lambda}(K'\times g^*)$ to
           a point in $\theta_{\lambda}(K'\times (g^* + \mathcal{H}_{\xi_{\lambda}}
           ([\Gamma])))$, where $\mathcal{H}_{\xi_{\lambda}}
           ([\Gamma])$ is the monodromy of $\Gamma$ with respect to the covering
           $\xi_{\lambda}: M(V^n,\lambda) \rightarrow Q^n$ (see~\eqref{Equ:Classify-Monodromy}).
       So $\theta_{\lambda}(K'\times g^*)$ and
       $\theta_{\lambda}(K'\times (g^* + \mathcal{H}_{\xi_{\lambda}}
           ([\Gamma])))$
       are in the same connected component of $M(V^n,\lambda)$.
       \vskip .2cm

       Let $\Gamma_1, \cdots, \Gamma_k $ be some closed curves in
       $Q^n$ based at $q_0$ so that
       $\{ [\Gamma_1], \cdots, [\Gamma_k] \}\subset
        H_1(Q^n, \Z_2)$ is the dual basis of the
         $\{ [\Sigma_1], \cdots ,[\Sigma_k] \} \subset H_{n-1}(Q^n, \Z_2)$ under
        $\Z_2$-intersection form of $Q^n$. Then
       by~\eqref{Equ:Monodromy-Equal}, we have
        \begin{align} \label{Equ:Monodromy-Full}
         \{ \mathcal{H}_{\xi_{\lambda}} ([\Gamma]) \, | \ \forall\, \Gamma : [0,1]
        \rightarrow Q^n  \} &= \mathrm{Im}(
        \mathcal{H}_{\xi_{\lambda}}) \notag \\
           &= \langle \mathcal{H}_{\xi_{\lambda}} ([\Gamma_1]), \cdots,
           \mathcal{H}_{\xi_{\lambda}} ([\Gamma_k]) \rangle  \subset (\Z_2)^m  \notag \\
           &= \langle \lambda(P_1),\cdots, \lambda(P_k) \rangle =
           L_{\lambda}.
         \end{align}
         So any element of
      $L_{\lambda}$ can be realized by
      $\mathcal{H}_{\xi_{\lambda}} ([\Gamma])$ for some closed curve $\Gamma$.
      Then the above claim is proved.
         \vskip .2cm

         So $\theta_{\lambda}(K\times g)$ and
      $\theta_{\lambda}(K'\times g')$ belong to the same connected
      component of $M(V^n,\lambda)$
        $\Longleftrightarrow$ $g' \in g + L_{\lambda}$. Since
        $\dim_{\Z_2} L_{\lambda} = \mathrm{rank}(\lambda) $,
        each connected component of $M(V^n,\lambda)$
        is made up of $2^{\mathrm{rank}(\lambda)}$
         copies of $V^n $ from $V^n\times (\Z_2)^m$. Indeed, suppose
         $(\Z_2)^m= L_{\lambda} \oplus \langle \omega_1 \rangle
         \oplus\cdots \oplus \langle\omega_{q} \rangle $ where
         $q=m-\mathrm{rank}(\lambda)$. Then
          $M(V^n,\lambda)$ has $2^{m-\mathrm{rank}(\lambda)}$
           connected components which are given by:
           $$\theta_{\lambda}\left(
              V^n\times \left(L_{\lambda} + t_1\omega_1 +
              \cdots t_q\omega_q \right) \right),\, t_i\in\{ 0 ,1 \}, 1\leq i \leq q, $$
            each of which is equipped with a natural free action by
     $L_{\lambda} \cong (\Z_2)^{\mathrm{rank}(\lambda)}$ defined in~\eqref{Equ:FreeAction} whose
     orbit space is $Q^n$.
     \end{proof}
     \vskip .2cm

        \begin{rem}
         In the above proof, let
          $\kappa: (\Z_2)^m \rightarrow L_{\lambda}\cong (\Z_2)^{\text{rank}(\lambda)}$ be
          a quotient homomorphism. Then for any $(\Z_2)^m$-coloring $\lambda$ on $V^n$,
         we can think of $\kappa\circ \lambda$ as a $(\Z_2)^{\text{rank}(\lambda)}$-coloring on
           $V^n$. It is easy to see that
           each connected component $K$ of $M(V^n,\lambda)$ is homeomorphic to
            $M(V^n, \kappa\circ \lambda)$.
        \end{rem}
        \vskip .2cm

     In general, for $I=\{ i_1, \cdots, i_s\} \subset \{ 1,\cdots, k \} $,
      the submanifold $\Sigma_I\subset Q^n$ may not be connected. Suppose $S$ is a
      connected component of $\Sigma_I$.
      Then $S$ is an $(n-s)$-dimensional
      embedded submanifold of $Q^n$. We also want to compute the number
      of components in $\xi^{-1}_{\lambda}(S)$, where
      $\xi_{\lambda} : M(V^n,\lambda) \rightarrow
         Q^n$ is the quotient map defined by~\eqref{Glue_Back}. \vskip .2cm

           \begin{rem}
       Two connected components $S$ and $S'$ of $\Sigma_I$ may not be homeomorphic to each other,
        and the
       $\xi_{\lambda}^{-1}(S)$ and $\xi_{\lambda}^{-1}(S')$ may not have
       equal number of connected components either.
     \end{rem}
     \vskip .2cm

      If we cut
        $S$ open along the transversely intersected embedded
        submanifolds $ \{ \Sigma_j \cap S \neq \varnothing \, |\, j\notin I
        \}$, we will get a nice manifold with corners, denoted by $V_S$.
        Similar to the $\Z_2$-core of $Q^n$, we can construct a (perfect)
        involutive panel structure on the
         boundary of $V_S$ from the cut sections of $ \{ \Sigma_j \cap S \neq \varnothing \, |\, j\notin I \}$. And any $(\Z_2)^m$-coloring $\lambda$ on $V^n$ induces a
        $(\Z_2)^m$-coloring $\lambda^{in}_{S}$ on the panels of $V_S$ by:
         if $E$ is the panel of $V_S$ corresponding to $\Sigma_j \cap S$,
           $$ \quad \lambda^{in}_{S} (E) := \lambda(P_j). $$
         It is easy to see that
         the glue-back construction $M(V_S,\lambda^{in}_S)$ is homeomorphic to
          $\xi_{\lambda}^{-1}(S)$. But $V_S$ may not be a $\Z_2$-core of $S$, since
         the homology classes of $ \{ \Sigma_j \cap S \neq \varnothing \, |\, j\notin I \}$ may not form a
         basis of $H_{n-s-1}(S,\Z_2)$. So we can not directly apply
          the formula in Theorem~\ref{thm:comp}
         to compute the number of components of $M(V_S,\lambda_S)$. In fact,
         the number of components of $M(V_S,\lambda_S)$ also depends
         on what homology classes are represented by
         $\{ \Sigma_j \cap S \neq \varnothing \, |\, j\notin I \}$ in $H_{n-s-1}(S,\Z_2)$.
         So we need to modify the proof of
          Theorem~\ref{thm:comp} to deal with this case. In the
          following, we will treat this problem in a very
          general setting on $Q^n$.
            \vskip .2cm

            Suppose $\{ N_1,\cdots, N_r \}$
            is an arbitrary collection of codimension one embedded
            submanifolds of a closed manifold $Q^n$ which lie in general position.
           Cutting $Q^n$ open along $N_1,\cdots, N_r$ gives us a
           nice manifold with corners $W^n$. As before, we can
           construct a (perfect) involutive panel structure
           $\widehat{P}_1,\cdots, \widehat{P}_r$ on the boundary
           of $W^n$ from the cut sections of $N_1,\cdots, N_r$.
           In addition, suppose $\Gamma_1, \cdots, \Gamma_k$ are simple closed curves
           in $Q^n$ whose homology classes form a basis of $H_1(Q^n,
           \Z_2)$, and each $\Gamma_i$ intersects $\{ N_1,\cdots, N_r \}$
           transversely. Let $a_{ij} \in \Z_2 $ be the mod 2 intersection number between
           $\Gamma_i$ and $N_j$. Note that for any fixed $j$, the $\{ a_{ij} \}$
           is completely determined by the homology class of $N_j$ in $H_{n-1}(Q^n,\Z_2)$. Indeed,
           if $\{ \Sigma_1, \cdots, \Sigma_k \}$ is a $\Z_2$-cut system of $Q^n$ whose
           homology classes is a dual basis of $\{ [\Gamma_1],\cdots, [\Gamma_k] \}$ under
           the $\Z_2$-intersection form, then for each $1\leq j\leq r$,
            the homology class $[N_j] = \sum_i a_{ij} [\Sigma_i] \in
            H_{n-1}(Q^n)$.
           \vskip .2cm

           For any $(\Z_2)^m$-coloring $\lambda :
           \{ \widehat{P}_1,\cdots, \widehat{P}_r \} \rightarrow (\Z_2)^m $ on $W^n$,
           we can show that $M(W^n,\lambda)$ is a closed manifold
           with a natural free $(\Z_2)^m$-action defined by~\eqref{Equ:FreeAction} whose orbit
           space is $Q^n$.
           The proof of this fact is exactly the same as Theorem~\ref{thm:manifold}, hence omitted.
            In addition, by a similar argument as in the proof Theorem~\ref{Thm:Equiv-Isom},
            the monodromy of each $\Gamma_i$
            with respect to the covering $M(W^n,\lambda) \rightarrow Q^n$ is given by:
            \begin{equation} \label{Equ:Monodromy-2}
              \widehat{\mathcal{H}}_{\lambda}(\Gamma_i) :=
            \sum^r_{j=1} a_{ij} \lambda(\widehat{P}_j) \in (\Z_2)^m.
            \end{equation}
            Now, we associate a new $(\Z_2)^m$-coloring $\widehat{\lambda}$
            to $\lambda$ on the panels of $W^n$ by:
              \[   \widehat{\lambda} (\widehat{P}_i) :=
               \widehat{\mathcal{H}}_{\lambda}(\Gamma_i),\ 1\leq i \leq r. \]
            \[ \text{Let}\  L_{\widehat{\lambda}} = \text{the subgroup of}\
            (\Z_2)^m \ \text{generated by}\  \{ \widehat{\lambda} (\widehat{P}_1) , \cdots,
            \widehat{\lambda} (\widehat{P}_r) \} .  \]
              \vskip .2cm

          \begin{thm} \label{thm:comp-1}
              For any $(\Z_2)^m$-coloring $\lambda$ on the panels of $W^n$ here,
              the number of connected components of $M(W^n,\lambda)$ equals
              $2^{m - l}$, where $l= \dim_{\Z_2}
              L_{\widehat{\lambda}}$.
             In addition, all the connected components of
             $M(W^n,\lambda)$ are homeomorphic to each other, and there is free
              $ L_{\widehat{\lambda}} \cong (Z_2)^l$-action on each component of
              $M(W^n,\lambda)$ whose orbit space is $Q^n$.
          \end{thm}
       \begin{proof}
           The argument here is parallel to that in
          Theorem~\ref{thm:comp} except that in~\eqref{Equ:Monodromy-Full}, the monodromy of
          $\Gamma_i$
          should be replaced by
           $\widehat{\mathcal{H}}_{\lambda}(\Gamma_i)$ in~\eqref{Equ:Monodromy-2}.
           So the proof is left to the reader.
       \end{proof}

   \vskip .7cm

  \section{Generalize to compact manifolds with boundary} \label{Sec4}

   We can generalize the notion of $\Z_2$-core and glue-back
   construction to any compact manifold with boundary.
   Suppose $X^n$ is an $n$-dimensional compact connected nice manifold with
       corners and $H^1(X^n,\Z_2) \neq 0$.
       Let $\{ F_1,\ldots, F_l \}$ be the set of facets of $X^n$.
        A $\Z_2$-\textit{cut system} of
      $X^n$ is a collection of $(n-1)$-dimensional embedded submanifolds
       $\Sigma_1 ,\cdots, \Sigma_k $ (possible with boundary)
       of $X^n$ which satisfy:
         \begin{enumerate}
          \item[(i)] $\Sigma_1 ,\cdots, \Sigma_k $ are in general position
       in $X^n$; and \vskip .1cm

          \item[(ii)]  the (relative) homology classes $ [\Sigma_1] ,\cdots, [\Sigma_k] $
       form a $\Z_2$-linear basis of
       $H_{n-1}(X^n, \partial X^n, \Z_2)\cong H^1(X^n,\Z_2)\neq
       0$.
       \end{enumerate}
            \vskip .2cm
        Moreover, we can choose each $\Sigma_i$ to be connected.
      If we cut
      $X^n$ open along $\{ \Sigma_1,\ldots, \Sigma_k \}$, we get a
      nice manifold with corners $U^n$, called a \textit{$\Z_2$-core} of $X^n$.
       Similar to Theorem~\ref{Thm:Conn_Z2-core}, we can show that
       there always exists a connect $\Z_2$-core for $X^n$.
      Note that the boundary stratification of $U^n$
      is a mixture of the facets in
      the cut section of $\Sigma_i$ and the facets from $\partial X^n$.
       So we define the panel structure on $U^n$
      by $\{ P_1,\cdots, P_k, P'_1, \ldots, P'_l \}$, where
      $P_i$ consists of the facets in the cut section of $\Sigma_i$ and $P'_j$
      consists of the facets in the cut open image of $F_j$.\vskip .2cm

      By a similar argument as in Lemma~\ref{Lem:Deform_Invol}, we can
      construct
      a free involution $\tau_i$ on each $P_i$ ($1\leq i \leq k$) which satisfies the
      conditions (a), (b)and (c) in Definition~\ref{Def:Invol_Panel}.
      If we do not define any involution on $P'_j$, we say
      that $\{ \tau_i: P_i \rightarrow P_i \}_{1\leq i \leq k} $ along with
      $\{P'_1, \ldots, P'_l  \} $ is a \textit{partial involutive panel structure} on the
      boundary of $U^n$. \vskip .2cm

          Let $\mathcal{P}(U^n) = \{P_1,\cdots, P_k\}$
           be the set of all panels in $U^n$ that are equipped with involutions.
         Any map from $\mathcal{P}(U^n)$ to
           $(\Z_2)^m$ is called a \textit{$(\Z_2)^m$-coloring on}
           $U^n$.
              It is easy to see that the glue-back construction
              $M(U^n, \lambda)$
      makes perfect sense for $U^n$ with a $(\Z_2)^m$-coloring
      $\lambda$. Indeed, $M(U^n, \lambda)$ is got by
       glue $2^m$ copies of $U^n$ only along
        those panels equipped with involutions
      according to the rule in~\eqref{Glue_Back}.\vskip .2cm

           By a parallel argument as Theorem~\ref{thm:manifold}, we can
           show that~\eqref{Equ:FreeAction} defines a natural free $(\Z_2)^m$-action
            on $M(U^n,\lambda)$ whose orbit space is homeomorphic to
        $X^n$. And similarly, we can prove the following.\vskip .2cm

     \begin{thm} \label{thm:bundle-color-2}
       Suppose $X^n$ is a compact connected nice manifold with corners and $U^n$
       is a $\Z_2$-core of $X^n$. Then we have:
   \begin{enumerate}
      \item  any principal $(\Z_2)^m$-bundle $\pi: M^n \rightarrow X^n$ determines
        a $(\Z_2)^m$-coloring $\lambda_{\pi}$ on $U^n$.
\vskip .2cm

     \item there is an equivariant homeomorphism
          from the $M(U^n, \lambda_{\pi})$ to $M^n$ which covers the identity of
          $X^n$.
   \end{enumerate}

     \end{thm}
     \vskip .2cm

  In addition, we can similarly define $L_{\lambda}$ and
  $\mathrm{rank}(\lambda)$ for any $(\Z_2)^m$-coloring $\lambda$ on $U^n$ (see~\eqref{Equ:L-lambda}
  and~\eqref{Equ:rank-lambda}) and
  extend the results in Theorem~\ref{thm:comp}
   to $M(U^n,\lambda)$ as well. Here we only give the
  statement below and leave the proof to the reader.\vskip .2cm

       \begin{thm}
        For any $(\Z_2)^m$-coloring $\lambda$ on $U^n$,
     the $M(U^n,\lambda)$ has $2^{m-\mathrm{rank}(\lambda)}$ connected components
     which are pairwise homeomorphic. Let $\theta_{\lambda}:
     U^n\times (\Z_2)^m \rightarrow M(U^n,\lambda)$ be the
       quotient map. Then each
       connected component of $M(U^n,\lambda)$ is homeomorphic to
       $\theta_{\lambda}(U^n\times L_{\lambda})$.
        And there is a free action of $L_{\lambda} \cong (\Z_2)^{\mathrm{rank}(\lambda)}$ on
       each connected component of $M(U^n,\lambda)$
       whose orbit space is $X^n$.
      \end{thm}

   \vskip .6cm

\section{Locally Standard $(\Z_2)^m$-action on closed $n$-manifolds}
\label{Sec5}

   First, let us define the meaning of \textit{standard action} of
   $(\Z_2)^m$ in dimension $n$ for any $m\geq 1$.  Suppose
   $g=(g_1,\cdots, g_m)$ is an arbitrary element of $(\Z_2)^m$.
   \begin{enumerate}
     \item  If $m\leq n$, the standard $(\Z_2)^m$-action on $\R^n$ is:
     \[ (x_1,\cdots, x_n) \longmapsto
     \left( (-1)^{g_1} x_1, \cdots, (-1)^{g_m} x_m, x_{m+1},\cdots, x_n \right), \]
   whose orbit space is $\R^{n,m}_+ := \{ (x_1,\ldots, x_n) \, ; \,
    x_i\geq 0\ \text{for}\ \forall\,
     1\leq i \leq m  \}$. \vskip .2cm

     \item
       For $m> n$,
      the standard $(\Z_2)^m$-action on $\R^n \times (\Z_2)^{m-n}$
      is:
      \begin{align*}
      & \qquad\qquad\ ((x_1,\cdots, x_n),(h_1,\cdots,h_{m-n}))
             \longmapsto \\
      & \left( ( (-1)^{g_1} x_1, \cdots, (-1)^{g_n} x_n ),
       (g_{n+1} + h_1, \cdots, g_m+h_{m-n}) \right),
       \end{align*}
      whose orbit space is $\R^{n,n}_+$.
 \end{enumerate}
 \vskip .2cm

    Suppose $(\Z_2)^m$ acts effectively and smoothly on a closed $n$-manifold $M^n$.
    A \textit{local isomorphism} of $M^n$ with the standard action (defined above) consists
    of: \vskip .2cm

   \begin{enumerate}
     \item a group automorphism $\sigma: (\Z_2)^m \rightarrow (\Z_2)^m$;

     \item a $(\Z_2)^m$-stable open set $V$ in $M^n$ and $U$ in
     $\R^n$ (if $m\leq n$) or $\R^n\times (\Z_2)^{m-n}$ (if $m > n$);

     \item a $\sigma$-equivariant homeomorphism $f : V \rightarrow
     U$, i.e. $f(g\cdot v) = \sigma(g)\cdot f(v)$ for any $g\in (\Z_2)^m$ and
     $v\in V$.
   \end{enumerate}
   \vskip .2cm

   A $(\Z_2)^m$-action on $M^n$ is called \textit{locally standard} if each point
   of $M^n$ is in the domain of some local isomorphism. Then
   $M^n$ is called a \textit{locally standard $(\Z_2)^m$-manifold} over $M^n\slash (\Z_2)^m$.
   Note that here
   we generalize the notion of \textit{locally standard $2$-torus manifold} defined
   in~\cite{MaLu08} where $m$ is required to be equal to $n$. \vskip .2cm

    Now, suppose we have a locally standard $(\Z_2)^m$-action on a closed manifold $M^n$.
    Then the orbit
   space $X^n = M^n\slash (\Z_2)^m$ is a nice manifold with corners
   (in the rest of the paper, we always assume $X^n$ is connected).
   Let $\pi : M^n \rightarrow X^n$ be the orbit map.
    Suppose the set of all facets of $X^n$ is
     $\mathcal{F}(X^n) = \{ F_1,\cdots, F_l\,  \}$.
   The \textit{characteristic function} $\nu_{\pi} : \mathcal{F}(X^n) \rightarrow (\Z_2)^m$
    of the action is defined by:
       \[ \nu_{\pi} (F_j) = \text{the element of}\ (\Z_2)^m \ \text{that fixes}\ \pi^{-1}(F_j)\
        \text{pointwise}.  \]
      Observe that whenever $F_{j_1} \cap \cdots \cap F_{j_s} \neq \varnothing$,
       $ \{\nu_{\pi}(F_{j_1}), \cdots, \nu_{\pi}(F_{j_s}) \} $ should be linearly
     independent vectors in $(\Z_2)^m$ over $\Z_2$. And the isotropy group of
     the set $ \pi^{-1}(F_{j_1}\cap \cdots \cap F_{j_s} ) $
     is the subgroup of $(\Z_2)^m$ generated by
     $\{ \nu_{\pi}(F_{j_1}), \cdots, \nu_{\pi}(F_{j_s}) \}$. \vskip .2cm

      In addition,
      $M^n$ determines a principal
      $(\Z_2)^m$-bundle over $X^n$, denoted by $\xi_{\pi}$.
      If $H^1(X^n,\Z_2) =0$, $\xi_{\pi}$ is always trivial. So
      we assume $H^1(X^n,\Z_2) \neq 0$ in the rest of this section. \vskip .2cm

     \begin{rem}
      If $m < n$, the dimension of any face of $X^n$
      is at least $n-m$. \vskip .3cm
     \end{rem}

     We will see that the characteristic function
     $\nu_{\pi}$ and the principal $\xi_{\pi}$ encode all the structural information of
      the $(\Z_2)^m$-action, and we can classify
      locally standard $(\Z_2)^m$-manifolds $\pi: M^n \rightarrow X^n$
      by $\nu_{\pi}$ and $\xi_{\pi}$
      up to some natural equivalence relations. The following discussions are parallel
      to those in~\cite{MaLu08}. \vskip .2cm

     First of all, we say that two locally standard $(\Z_2)^m$-manifolds
   $M^n$ and $N^n$ over $X^n$ are \textit{equivalent} if there is a homeomorphism
   $f: M^n \rightarrow N^n$ together with an element $\sigma \in \mathrm{GL}(m, \Z_2)$
   such that
   \begin{enumerate}
     \item  $f(g\cdot x) = \sigma(g)\cdot f(x)$ for all $g\in (\Z_2)^m$ and $x \in M^n$,
     and \vskip .1cm
     \item  $f$ induces the identity map on the orbit space.
     \vskip .2cm
   \end{enumerate}

    In addition, we call two locally standard $(\Z_2)^m$-manifolds
   $M^n$ and $N^n$ over $X^n$ \textit{equivariantly homeomorphic} if there is a homeomorphism
   $f: M^n \rightarrow N^n$ such that
   $f(g\cdot x) =g \cdot f(x)$ for all $g\in (\Z_2)^m$ and $x \in M^n$.
   Such an $f$ is called an \textit{equivariant homeomorphism} between
   $M^n$ and $N^n$. Notice that $f$ will induce a
     homeomorphism $h_f: X^n \rightarrow X^n$ which preserves the manifold with corners
     structure of $X^n$. But $h_f$ may not be the identity map of $X^n$.
     \vskip .2cm

   Let $U^n$ be a $\Z_2$-core of $X^n$ from cutting $X^n$ open along a $\Z_2$-cut system
   $\{ \Sigma_1,\cdots, \Sigma_k \}$ in $X^n$ as described in Section~\ref{Sec4}.
   The panel structure of $U^n$ is $\{ P_1,\cdots, P_k, P'_1, \cdots, P'_l \}$ where
      $P_i$ corresponds to the cut section of $\Sigma_i$ and $P'_j$
      is the cut open image of the facet $F_j$, and there is a
       partial involutive panel structure $\tau_i : P_i \rightarrow P_i$ on
       $U^n$.
       Moreover,
      if we define $\tau'_j = id : P'_j \rightarrow
      P'_j$ for any $1\leq j \leq l$, then $\{ P_i, \tau_i \}_{1\leq i \leq
      k}$ and $\{ P'_j, \tau'_j \}_{1\leq j \leq l}$ together define
      a complete involutive panel structure on $U^n$. We call $\{ P_1,\cdots, P_k \}$ the
   \textit{principal panels} of $U^n$ and call $\{  P'_1, \cdots, P'_l \} $
      the \textit{reflexive panels} of $U^n$. And we
      assume $U^n$ having this involutive panel structure
      in the rest of this paper.
       \vskip .2cm

        By Theorem~\ref{thm:bundle-color-2}, the principal bundle $\xi_{\pi}$
      determines a $(\Z_2)^m$-coloring $\lambda_{\pi}$ on the set of principal panels
      $\{ P_1, \cdots, P_k \}$, and
      the characteristic function $\nu_{\pi}$ induces a $(\Z_2)^m$-coloring
      $\mu_{\pi}$ on the reflexive panels $\{  P'_1, \cdots, P'_l \} $ by
       $\mu_{\pi}(P'_j) = \nu_{\pi}(F_j)$, $1\leq j \leq l$. So
       for any $P'_{j_1} \cap \cdots \cap P'_{j_s} \neq
       \varnothing$, we should have:
      \begin{equation} \label{Equ:Linear-Indep}
         \mu_{\pi}(P'_{j_1}), \cdots, \mu_{\pi}(P'_{j_s})
        \ \text{is linearly independent vectors in}\ (\Z_2)^m \
        \text{over}\ \Z_2 .
        \end{equation}

       The glue-back construction of $U^n$ with respect to the
      composite $(\Z_2)^m$-coloring $ ( \lambda_{\pi}, \mu_{\pi} ) $ on the panels of
      $U^n$ gives us a closed manifold, denoted by $M(U^n, \lambda_{\pi},\mu_{\pi} )$. The natural
      $(\Z_2)^m$-action on  $M(U^n, \lambda_{\pi},\mu_{\pi} )$ is also defined
      by:
       \begin{equation} \label{Equ:Local-Standard-Action}
              g\cdot [(x,g_0)] := [(x, g+g_0)],\; \forall\, x\in U^n, \
             \forall\, g, g_0\in (\Z_2)^m .
           \end{equation}
          \vskip .2cm

       The following theorem is parallel to that in~\cite{MaLu08}. \vskip .2cm

      \begin{thm} \label{Thm:Equiv-Isom_2}
        The action  $ (\Z_2)^m \curvearrowright M(U^n, \lambda_{\pi},\mu_{\pi} )$
        defined in~\eqref{Equ:Local-Standard-Action} is locally standard
        and there is an equivariant homeomorphism from
        $M(U^n, \lambda_{\pi},\mu_{\pi} )$ to $M^n$ which covers the
        identity of $X^n$.
      \end{thm}
      \begin{proof}
           It is easy to check the action is locally standard. And
           by a parallel argument as the proof of Theorem~\ref{thm:manifold},
           the orbit space of
             $(\Z_2)^m \curvearrowright M(U^n, \lambda_{\pi},\mu_{\pi} )$
             is $X^n$.
            Moreover,
            $(\Z_2)^m \curvearrowright M(U^n, \lambda_{\pi},\mu_{\pi} )$
            defines the same principal $(\Z_2)^m$-bundle over
             $X^n$
            and the same characteristic function on
            $X^n$ as the locally standard $(\Z_2)^m$-action on $M^n$.
          Then it is easy to construct an equivariant homeomorphism from
          $M(U^n, \lambda_{\pi},\mu_{\pi} )$ to $M^n$ which covers the identity of
          $X^n$.
      \end{proof}

      Denote by $\Xi(U^n, (\Z_2)^m )$ the set of all eligible composite
   $(\Z_2)^m$-colorings on $U^n$,\vskip .1cm

    $  \Xi(U^n, (\Z_2)^m) := \{ (\lambda, \mu) \, | \, \lambda \
     \text{is a}\ (\Z_2)^m\text{-coloring on the principal panels }  \text{of}\ U^n;  $

    $ \qquad\qquad\qquad\qquad\qquad\ \ \mu\ \text{is a}\ (\Z_2)^m\text{-coloring on the
       reflexive panels}\ \text{of}\ U^n  $

    $\qquad\qquad\qquad\qquad\qquad\ \ \text{which}  \
    \text{satisfies the condition~\eqref{Equ:Linear-Indep}}
    \}$. \vskip .2cm

   Then by Theorem~\ref{Thm:Equiv-Isom_2}, any locally standard $(\Z_2)^m$-action
   on a closed $n$-manifold with
   $X^n$ as the orbit space can be obtained from
   $U^n$ and some composite $(\Z_2)^m$-coloring $(\lambda,\mu) \in  \Xi(U^n, (\Z_2)^m)
   $. \vskip .2cm

   By the definition of $U^n$, it is easy to see that
    we can identify $\Xi(U^n, (\Z_2)^m)$ with
   $H^1(X^n,(\Z_2)^m) \times \mathcal{V}(X^n,(\Z_2)^m)$ as a set, where
   $\mathcal{V}(X^n,(\Z_2)^m)$ is the set of all
   characteristic functions on $X^n$, i.e.
   \vskip .2cm

    $  \mathcal{V}(X^n,(\Z_2)^m) := \{  \nu : \mathcal{F}(X^n) \rightarrow (\Z_2)^m \, ; \,
         \nu(F_{j_1}), \cdots, \nu(F_{j_s})  \ \text{are
         linearly } $

     $\qquad\qquad\qquad\quad
     \text{independent vectors in}\ (\Z_2)^m\ \text{whenever}\ F_{j_1}
     \cap \cdots \cap F_{j_s} \neq \varnothing \} $. \vskip .5cm

  \begin{exam}
    Suppose $P^n$ is a convex simple polytope with $m$ facets
    $\{ F_1,\ldots, F_m \}$. Let $\{ e_1, \ldots, e_m \}$ be a
    basis of $(\Z_2)^m$. If we color $F_i$ by $e_i$, the glue-back
    construction for $P^n$ with the trivial involutive panel
    structure (see Example~\ref{Exam:Trivial-Panel})
    gives us a manifold $\mathcal{Z}_{P^n}$, called the \textit{real moment-angle manifold}
    over $P^n$ (see~\cite{DaJan91} and~\cite{BP02}).
    The $\Z_2$-coefficient equivariant cohomology
     ring of $\mathcal{Z}_{P^n}$ with respect to the natural $(\Z_2)^m$ action
     is isomorphic to the face ring of $P^n$
     (see~\cite{DaJan91}). The ordinary $\Z_2$-cohomology groups
     of $\mathcal{Z}_{P^n}$
     were calculated by XiangYu Cao and Zhi L\"u in~\cite{LuZhi09}.\vskip .2cm
  \end{exam}

  \begin{exam}
     In Figure~\ref{p:GB-Pentagon}, we have three different
     $(\Z_2)^3$-colorings of a pentagon which is
     equipped with the trivial involutive panel structure (see Example~\ref{Exam:Trivial-Panel}). The
     glue-back construction for the left picture gives $T^2 \# T^2$ (connected
     sum of two tori). For the other two pictures, the glue-back constructions both give
     the connected sum of two Klein bottles (these examples are taken from~\cite{LuYu07}).

      \begin{figure}
         \includegraphics[width=0.53\textwidth]{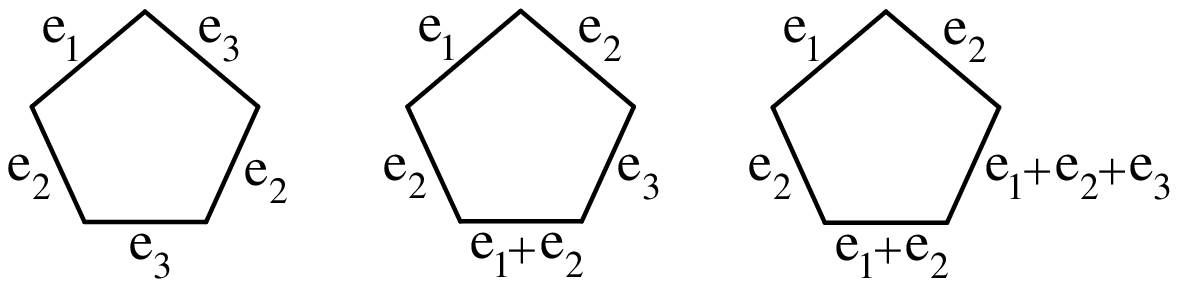}
          \caption{  }\label{p:GB-Pentagon}
      \end{figure}

  \end{exam}

  \begin{exam}
    Let $X^2= T^2 - P^2$ where $P^2$ is a polygon on $T^2$.
   In Figure~\ref{p:Z2-core-2}, we have three different $\Z_2$-cores of $X^2$.
    Then any locally standard $(\Z_2)^m$-action on a closed manifold
     with $X^2$ as the orbit space can be obtained by the
     glue-back construction from any one of the three $\Z_2$-cores with some suitable
     composite $(\Z_2)^m$-coloring.
  \end{exam}
   \begin{figure}
         \includegraphics[width=0.64\textwidth]{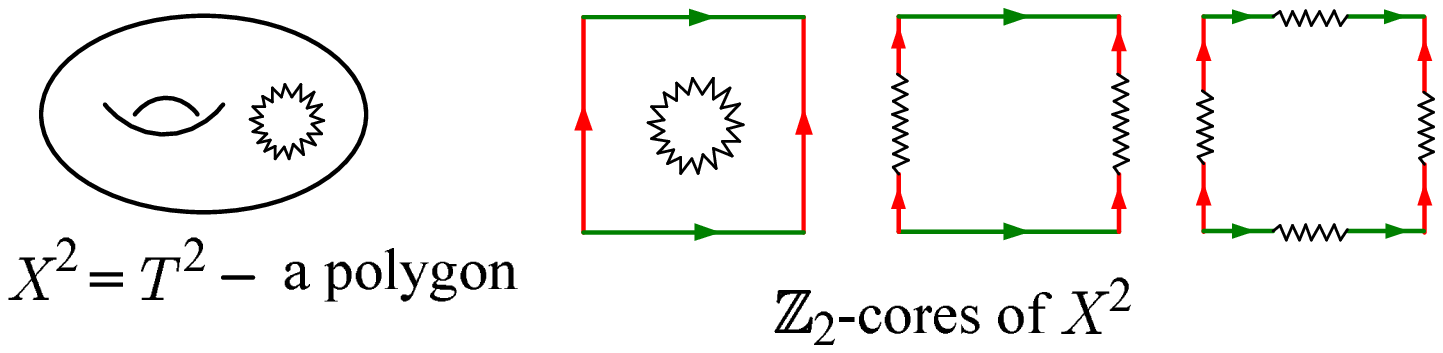}
          \caption{  }\label{p:Z2-core-2}
      \end{figure}

    Next, let us classify the locally standard $(\Z_2)^m$-manifolds over $X^n$ up to
    the equivalence from the
    viewpoint of glue-back construction.
    By a similar argument as in~\cite{MaLu08}, we can prove:

   \begin{thm}
     The set of equivalence classes in locally standard $(\Z_2)^m$-manifolds
   over $X^n$ bijectively corresponds to the coset
     $ \Xi(U^n, (\Z_2)^m) \slash \mathrm{GL}(m,\Z_2)$, where the
   $\mathrm{GL}(m,\Z_2)$ acts on $\Xi(U^n, (\Z_2)^m)$
   via automorphisms of the coefficient group $(\Z_2)^m$.

\end{thm}
 \vskip .2cm

     Moreover, similar to~\cite{MaLu08}, we can classify locally standard
   $(\Z_2)^m$-manifolds over $X^n$ up to equivariant homeomorphisms as
   following. Let $\mathrm{Aut}(X^n)$ be the group of self-homeomorphisms of $X^n$
      which preserve the manifold with corners structure of $X^n$.
       An element $h\in \mathrm{Aut}(X^n)$
   will induce a permutation on $\mathcal{F}(X^n)$ denoted by
     $\Phi(h): \mathcal{F}(X^n) \rightarrow
   \mathcal{F}(X^n)$.
    So $h$
   naturally acts on $\mathcal{V}(X^n,(\Z_2)^m)$ by sending any $\nu \in \mathcal{V}(X^n,(\Z_2)^m)$
   to $\nu\circ \Phi(h)$.
   \vskip .2cm

   \begin{thm}
   Suppose $\pi: M^n\rightarrow X^n$ and $\pi' : N^n\rightarrow X^n$ are
    two locally standard $(\Z_2)^m$-manifolds over $X^n$.
  $M^n$ and $N^n$ are equivariantly homeomorphic if and only if there is an $h\in \mathrm{Aut}(X^n)$
  such that $\nu_{\pi'} = \nu_{\pi}\circ \Phi(h)$ and $h^*(\xi_{\pi'})= \xi_{\pi}$
  where $h^*(\xi_{\pi'})$ is the pull-back bundle by $h$.
   \end{thm}

  \begin{thm}
    The set of equivariant homeomorphism classes of all $n$-dimensional locally standard
   $(\Z_2)^m$-manifolds over $X^n$ bijectively corresponds to the coset
    \[   \left( H^1(X^n,(\Z_2)^m) \times \mathcal{V}(X^n,(\Z_2)^m) \right) \slash \mathrm{Aut}(X^n)   \]
    where $\mathrm{Aut}(X^n)$ acts diagonally on the two factors.
    \vskip .2cm
  \end{thm}

   In addition, we say that
    two locally standard $(\Z_2)^m$-manifolds $M^n$ and $N^n$ over $X^n$ are
    \textit{weakly equivariantly homeomorphic} if there is a homeomorphism
   $f: M^n \rightarrow N^n$ and an element $\sigma \in \mathrm{GL}(m, \Z_2)$
   such that $f(g\cdot x) = \sigma(g)\cdot f(x)$ for all $g\in (\Z_2)^m$ and $x \in M^n$.
   \vskip .2cm

    \begin{thm}
    The set of weakly equivariant homeomorphism classes in all $n$-dimensional locally standard
   $(\Z_2)^m$-manifolds over $X^n$ bijectively corresponds to the double coset
    \[  \mathrm{GL}(m, \Z_2) \backslash \left( H^1(X^n,(\Z_2)^m)
    \times \mathcal{V}(X^n,(\Z_2)^m) \right) \slash \mathrm{Aut}(X^n)   \]
    where both $ \mathrm{GL}(m, \Z_2)$ and $\mathrm{Aut}(X^n)$ act diagonally on the two
    factors.
  \end{thm}
   \ \\

   \section{Some Topological Information of Locally Standard $(\Z_2)^m$-Manifolds
   from the $(\Z_2)^m$-colorings} \label{Sec6}

    Suppose $U^n$ is a $\Z_2$-core of a connected nice manifold with corners
    $X^n$, and the principal panels and
    reflexive panels of $U^n$ are $\{ P_1,\cdots, P_k \}$ and $\{ P'_1, \cdots, P'_l \}$
      as described in the preceding section. Similar to Theorem~\ref{thm:comp},
     we can compute the number of connected components in any $M(U^n, \lambda, \mu)$
     from the composite $(\Z_2)^m$-coloring $(\lambda, \mu)$ on $U^n$ as following.\vskip .2cm

    \begin{thm} \label{thm:comp-2}
        For any $(\lambda, \mu) \in \Xi(U^n,(\Z_2)^m)$,
       the number of connected components in $M(U^n, \lambda, \mu )$
        is $2^{m-\text{rank}(\lambda, \mu)}$ where
        \[          \text{rank}(\lambda, \mu) = \dim_{\Z_2}
          \langle \lambda(P_1),\cdots, \lambda(P_k),
         \mu(P'_1), \cdots, \mu(P'_l) \rangle.       \]
         The connected components of $M(U^n, \lambda, \mu )$
         are pairwise homeomorphic, and there is an induced
         locally standard $(\Z_2)^{\text{rank}(\lambda,
         \mu)}$-action on each component of $M(U^n, \lambda, \mu )$.
    \end{thm}
       \begin{proof}
      Suppose we glue
       the $2^m$ copies of $U^n$ only along the principal panels first
       according to the coloring $\lambda$, we
       will get a manifold with boundary denoted by $M(U^n,\lambda)$.
        By the same argument as in the proof of Theorem~\ref{thm:comp},
       $M(U^n,\lambda)$ has
       $2^{m-\text{rank}(\lambda)}$ connected components which are pairwise
        homeomorphic. Let
       $ L_{\lambda}:=  \langle \lambda(P_1),\cdots, \lambda(P_k)\rangle \subset (\Z_2)^m$
        and let $\theta_{\lambda} : U^n \times (\Z_2)^m \rightarrow M(U^n, \lambda)$
       be the quotient map.
         Then an arbitrary connected component of $M(U^n,\lambda)$ is of the following form
        \[ N_{g}= \bigcup_{ g' \in g + L_{\lambda}}
         \theta_{\lambda} ( U^n \times g' )\ \text{ for some
         fixed} \ g\in (\Z_2)^m \]

        So $M(U^n,\lambda) = N_{g_1} \cup \cdots \cup N_{g_r}$
        where $r= 2^{m-\text{rank}(\lambda)}$ and $g_{i'}- g_i \notin L_{\lambda}
        $ for any $1\leq i,i' \leq r$.
           The boundary of $M(U^n,\lambda)$ consists of those facets from the reflexive panels of
         $ U^n \times g's$. \vskip .2cm

         Next, we glue the facets
          in the boundary of $M(U^n,\lambda)$ together
         according to~\eqref{Glue_Back} and the coloring $\mu$ on the reflexive panels,
         which will give us the $M(U^n,\lambda, \mu)$. Let
         $\theta_{\mu} : M(U^n,\lambda) \rightarrow M(U^n,\lambda, \mu)$ denote
         the corresponding quotient map. In addition, let
         $L_{\mu} = \langle \mu(P'_1), \cdots, \mu(P'_l) \rangle \subset (\Z_2)^m $.
         It is easy to see that $\theta_{\mu}(N_{g_i})$ and $\theta_{\mu}(N_{g_{i'}})$ are
         in the same connected component of $M(U^n,\lambda, \mu)$
         if and only if $g_{i'} - g_i \in L_{\mu}$. So for any
         $g,g' \in (\Z_2)^m$, the two blocks $\theta_{\mu} \left( \theta_{\lambda}
         (U^n\times g) \right)$ and
         $\theta_{\mu} \left( \theta_{\lambda} (U^n\times g')\right)$ are in the
         same connected component of $M(U^n,\lambda, \mu)$ if and
         only if $g' - g \in L_{\lambda} + L_{\mu}$.
          Since the $\Z_2$-dimension of
         $L_{\lambda} + L_{\mu}$ is $\text{rank}(\lambda, \mu)$, so
         each connected component of $M(U^n,\lambda, \mu)$ is the
         gluing of $2^{\text{rank}(\lambda, \mu)}$ copies of $U^n$ from the glue-back construction.
         Then $M(U^n,\lambda, \mu)$ has exactly $2^{m-\text{rank}(\lambda,
         \mu)}$ connected components which are pairwise homeomorphic.
         Obviously, the restricted action
          of $(\Z_2)^m$ to $L_{\lambda} + L_{\mu}$ on
        each component of $M(U^n,\lambda, \mu)$ is locally standard.
        So our theorem is proved.
       \end{proof}
       \vskip .2cm

      In addition, if $X^n$ is orientable (in integral coefficient),
     using the same argument as the Theorem1.7 in~\cite{NaNish05}, we can prove the following.

  \begin{thm} \label{thm:orientation}
     For a basis $\{ e_1 ,\cdots, e_m \}$ of $(\Z_2)^m$, there is a group
    homomorphism $\epsilon : (\Z_2)^m \rightarrow \Z_2$ defined
    by $\epsilon(e_i) = 1$ for all $i$.
    Suppose $X^n$ is orientable. Then $M(U^n, \lambda, \mu)$ is orientable if
     and only if there exists a basis $\{ e_1 ,\cdots, e_m \}$ of $(\Z_2)^m$
    such that $  \epsilon(\mu(P'_1)) = \cdots = \epsilon(\mu(P'_l)) = 1 $.
    So in this case, the orientablity of $M(U^n, \lambda, \mu)$ is determined
     only by the coloring $\mu$ on the reflexive panels.
    \end{thm}

     \vskip .2cm

    It was shown in~\cite{DaJan91} that the $\Z_2$-coefficient cohomology ring of any small cover
    can be computed from
    the combinatorial structure of the orbit space and the associated
    characteristic function.
     But for a general locally standard $(\Z_2)^m$-manifold
    $M^n$ over $X^n$, it is not clear how to
    compute the $\Z_2$-homology group or cohomology ring of $M^n$
    from $X^n$. From our preceding discussions, the simplest case in this problem
    would be when $X^n$ has a $\Z_2$-core $U^n$ whose faces are all
    contractible. So we propose the following problem. \vskip .4cm

   \textbf{Problem:}  for a locally standard $(\Z_2)^m$-manifold
    $M(U^n, \lambda, \mu)$ where
     $(\lambda, \mu) \in \Xi(U^n, (\Z_2)^m)$, if each face of $U^n$ is
     contractible (or a $\Z_2$-homology ball),
     find some way to compute the $\Z_2$-homology group and cohomology ring of
    $ M(U^n, \lambda, \mu)$ from the data
    $(U^n, \lambda, \mu)$.

  \ \\

  \textbf{Acknowledgement:} I want to thank professor Zhi L\"{u}
    for showing me the paper~\cite{MaLu08} and a lot of patient
    explanation.

     \ \\

\end{document}